\documentclass[preprint,11pt]{elsarticle}

\usepackage{amsfonts,amsthm,amssymb}
\usepackage{enumerate}
\usepackage{placeins}
\usepackage[hmargin={28mm,28mm},vmargin={30mm,35mm}]{geometry}
\usepackage{relsize}
\usepackage{xcolor}
\usepackage[ocgcolorlinks,allcolors={blue}]{hyperref}
\usepackage{nicefrac}
\usepackage{mathtools,mathrsfs}
\usepackage{bm}
\usepackage{bbm}
\usepackage{mathabx}
\usepackage{setspace}
\usepackage{subfig}
\usepackage{floatrow}
\usepackage{color}
\usepackage[applemac]{inputenc}
\usepackage[T1]{fontenc}
\usepackage{url}
\usepackage{upgreek}
\usepackage{paralist}
\usepackage{accents}
\allowdisplaybreaks
\newcommand{\bbbracketleft}{[\hspace{-2.2pt}[}
\newcommand{\bbbracketright}{]\hspace{-2.2pt}]}
\newcommand{\curlybraceleft}{\{\hspace{-3.2pt}\{}
\newcommand{\curlybraceright}{\}\hspace{-3.2pt}\}}
\newcommand{\jm}[1]{{\bbbracketleft{#1}\bbbracketright}}
\newcommand{\av}[1]{{\curlybraceleft{#1}\curlybraceright}}
\newcommand{\tpn}[1]{ {\vvvert{#1}\vvvert} }
\newcommand{\sF}[1]{\mathsf{#1}}

\newcommand*{\dotp}{\raisebox{-0.2ex}
{\scalebox{1.3}{$\cdot$}}}
\renewcommand*{\ddot}[1]{%
  \accentset{\mbox{\footnotesize\bfseries .\hspace{-0.05ex}.}}{#1}}
 \renewcommand*{\dot}[1]{%
  \accentset{\mbox{\footnotesize\bfseries .}}{#1}}
\newcommand{\U}{{\bf U}}
\newcommand{\W}{{\bf W}}
\newcommand{\E}{{\cc E}}

\newcommand{\Del}{\triangle}

\newcommand{\dG}{{\mathrm{dG}}}

\newcommand{\dGe}{{\mathrm{dG},e}}
\newcommand{\dGa}{{\mathrm{dG},a}}

\newcommand{\wl}{\overline}
\newcommand{\la}{\left\langle}
\newcommand{\ra}{\right\rangle}

\newcommand{\bb}{\mathbb}

\let\temp\phi
\let\phi\varphi
\let\varphi\temp
\newcommand{\pphi}{\varphi}
\newcommand{\nf}{\nicefrac}
\newcommand{\cc}{\mathcal}
\newcommand{\de}{\partial}

\newcommand{\epsi}{\epsilon}
\newcommand{\eps}{\varepsilon}
\newcommand{\wh}{\widehat}
\newcommand{\z}{\zeta}

\renewcommand{\k}{\kappa}

\newcommand{\Oe}{{\Omega_e}}
\newcommand{\Oa}{{\Omega_a}}
\newcommand{\grad}  {\bf\nabla}
\newcommand{\gradh}{{\grad_{\! h}}}

\newcommand{\divb} {\mathbf{div}\,}		% Divergenza di un tensore (grassetto)
\renewcommand{\div} {\mbox{\rm{div}}\,}		% Divergenza di un vettore (minuscolo)

\renewcommand{\bf}{\mathbf}
\newcommand{\mr}{\mathrm}
\renewcommand{\d}{\mr{d}}
\DeclareMathOperator{\card}{card}

\newtheorem{theorem}{Theorem}
\numberwithin{theorem}{section}

\newtheorem{lemma}{Lemma}[section]
\newtheorem{corollary}[theorem]{Corollary}
\theoremstyle{remark}
\newtheorem{remark}[theorem]{Remark}
\theoremstyle{definition}
\newtheorem{assumption}{Assumption}

\makeatletter
\newenvironment{subtheorem}[1]{%
  \def\subtheoremcounter{#1}%
  \refstepcounter{#1}%
  \protected@edef\theparentnumber{\csname the#1\endcsname}%
  \setcounter{parentnumber}{\value{#1}}%
  \setcounter{#1}{0}%
  \expandafter\def\csname the#1\endcsname{\theparentnumber\alph{#1}}%
  \ignorespaces
}{%
  \setcounter{\subtheoremcounter}{\value{parentnumber}}%
  \ignorespacesafterend
}
\makeatother
\newcounter{parentnumber}
\makeatletter
\def\ps@pprintTitle{%
 \let\@oddhead\@empty
 \let\@evenhead\@empty
 \def\@oddfoot{}%
 \let\@evenfoot\@oddfoot}
\makeatother

\begin{document}

\begin{frontmatter}



\title{A high-order discontinuous Galerkin approach to the elasto-acoustic problem\tnoteref{label1}}

\tnotetext[label1]{This work has been supported by SIR Research Grant no.~RBSI14VTOS funded by MIUR -- Italian Ministry of Education, Universities, and Research.}


\author[1]{Paola F.~Antonietti}
\ead{paola.antonietti@polimi.it}
\author[1]{Francesco Bonaldi\corref{cor1}}
\ead{francesco.bonaldi@polimi.it}
%{Corresponding author, \email{francesco.bonaldi@polimi.it}}}
\author[1]{Ilario Mazzieri}
\ead{ilario.mazzieri@polimi.it}
\address[1]{MOX, Dipartimento di Matematica, Politecnico di Milano\\ Piazza Leonardo da Vinci 32, 20133 Milano, Italy}%
\cortext[cor1]{Corresponding author.} 
\begin{abstract}
\noindent
We address the spatial discretization of an evolution problem arising from the coupling of viscoelastic and acoustic wave propagation phenomena by employing a discontinuous Galerkin scheme on polygonal and polyhedral meshes.
The coupled nature of the problem is ascribed to suitable transmission conditions imposed at the interface between the solid (elastic) and fluid (acoustic) domains. We state and prove a well-posedness result for the strong formulation of the problem, present a stability analysis for the semi-discrete formulation, and finally prove an \emph{a priori} $hp$-version error estimate for the resulting formulation in a suitable (mesh-dependent) energy norm. We also discuss the time integration scheme employed to obtain the fully discrete system. The convergence results are validated by numerical experiments carried out in a two-dimensional setting.
\end{abstract}

\begin{keyword}
discontinuous Galerkin methods \sep elastodynamics \sep acoustics \sep wave propagation \sep polygonal and polyhedral grids


\MSC[2010] 65M12 \sep 65M60

\end{keyword}

\end{frontmatter}


\section*{Introduction}
This work is devoted to the development and analysis of a discontinuous Galerkin (dG) method on polygonal and polyhedral grids \cite{unified-brezzi, riviere, daniele-ern, hesthaven} for an evolution problem modeling the coupling of viscoelastic and acoustic wave propagation phenomena. Such kind of problems arise, for example, in geophysics, namely in the modeling and simulation of seismic events near coastal environments. Other contexts in which this problem plays a major role are the modeling of sensing or actuation devices immersed in an acoustic fluid \cite{flemisch-wohlmuth}, as well as medical ultrasonics \cite{monkola-sanna}. In practical applications, the underlying geometry one has to deal with is remarkably complicated and irregular; considering a conforming triangulation would therefore be computationally very expensive. We are thus led to consider a space discretization method capable to reproduce the geometrical constraints under consideration to a reasonable extent of accuracy, without being at the same time too much demanding. Such a discretization is then performed using general polygonal or polyhedral (briefly, \emph{polytopic}) elements, with no restriction on the number of faces each element can possess, and possibly allowing for face degeneration in mesh refinement. The dG method has been recently proven to successfully support polytopic meshes: we refer the reader, e.g., to \cite{antonietti-poligoni, antonietti-review, antonietti-mazzieri-polyhedra, antonietti-cmame, cangiani2014,cangiani2016,cangiani2017,pennesi1,pennesi2}, as well as to the comprehensive research monograph by Cangiani \emph{et al}.~\cite{cangiani-book}. In addition to the dG method, several other methods are capable to support polytopic meshes, such as the Polygonal Finite Element method \cite{sukumar1,sukumar2,sukumar3,sukumar4}, the Mimetic Finite Difference method \cite{antonietti-mfd, brezzi-mfd,gyrya-mfd,beirao-mfd}, the Virtual Element method \cite{vem1,vem2,vem3,vem4}, the Hybridizable Discontinuous Galerkin method \cite{hdg1,hdg2,hdg3,hdg4,hdg5}, and the Hybrid High-Order method \cite{daniele-cmame,daniele-lemaire,daniele-droniou,bonaldi-hho,tittarelli}.

An elasto-acoustic coupling typically occurs in the following framework: a domain made up by two subdomains, one occupied by a solid (elastic) body, the other by a fluid (acoustic) one, with suitable transmission conditions imposed at the interface between the two. The aim of such  conditions is to account for the following physical properties:~\begin{inparaenum}[(i)] \item the normal component of the velocity field is continuous at the interface;
\item a pressure load is exerted by the fluid body on the solid one through the interface.
\end{inparaenum}
In this paper, the unknowns of the problem are the displacement field in the solid domain and the acoustic potential in the fluid domain; the latter, say $\phi$, is defined in terms of the acoustic velocity field $\bf v_a$ in such a way that $\bf v_a = -\grad\phi$. However, other formulations are possible; for instance, one can consider a pressure-based formulation in the acoustic subdomain \cite{monkola-sanna}, or a displacement-based formulation in both subdomains \cite{bermudez}.

In a geophysics context, when a seismic event occurs, both \emph{pressure} (P) and \emph{shear} (S) waves are generated. However, only P-waves (i.e., whose direction of propagation is aligned with the displacement of the medium) are able to travel through both solid and fluid bodies, unlike S-waves (i.e., whose direction of propagation is orthogonal to the displacement of the medium), which can travel only through solids. This explains the reason for considering the first interface condition. On the other hand, the second one accounts for the fact that an acoustic wave propagating in a fluid domain of density $\rho_a$ gives rise to an \emph{acoustic pressure} field of magnitude $\rho_a |\dot\phi | $, $\dot\phi$ denoting the first time derivative of the acoustic potential.

Mathematical and numerical aspects of the elasto-acoustic coupling have been the subject of an extremely broad literature. We give hereinafter a brief overview of some of the research works carried out so far in this field.

Barucq \emph{et al}.~\cite{barucq-frechet} have characterized the Fréchet differentiability of the elasto-acoustic field with respect to Lipschitz-continuous deformation of the shape of an elastic scatterer. The same authors \cite{barucq-dg} have also proposed a dG method for computing the scattered field from an elastic bounded object immersed in an infinite homogeneous fluid medium, employing high-order polynomial-shape functions to address the high-frequency propagation regime, and curved boundary edges to provide an accurate representation of the fluid-structure interface. Berm\'udez \emph{et al.}~\cite{bermudez} have solved an interior elasto-acoustic problem in a three-dimensional setting, employing a displacement-based formulation on both the fluid and the solid domains, and a discretization consisting of linear tetrahedral finite elements for the solid and Raviart--Thomas elements of lowest order for the fluid; a further unknown is introduced on the interface between solid and fluid to impose the trasmission conditions. Brunner \emph{et al.}~\cite{brunner-junge-gaul} have treated the case of thin structures and dense fluids; the structural part is modeled with the finite element method, and the exterior acoustic problem is efficiently modeled with the Galerkin boundary element method. De Basabe and Sen \cite{debasabe-sen} have compared Finite Difference and Spectral Element methods for elastic wave propagation in media with a fluid-solid interface. Fischer and Gaul \cite{fischer-gaul} have proposed a coupling algorithm based on Lagrange multipliers for the simulation of structure-acoustic interaction; finite plate elements are coupled with a Galerkin boundary element formulation of the acoustic domain, and the interface pressure is interpolated as a Lagrange multiplier, thereby allowing for coupling non-matching grids. Flemisch \emph{et al.}~\cite{flemisch-wohlmuth} have considered a numerical scheme based on two independently generated grids on the elastic and acoustic domains, thereby allowing as much flexibility as possible, given that the computational grid in one subdomain can in general be considerably coarser than in the other subdomain. As a result, non-conforming grids appear at the interface of the two subdomains. Mandel \cite{mandel} has proposed a parallel iterative method for the solution of the linear equations resulting from the finite element discretization of the coupled fluid-solid systems in fluid pressure and solid displacement formulation, in harmonic regime. Mönköla \cite{monkola-sanna-paper} has examined the accuracy and efficiency of the numerical solution based on high-order discretizations, in the case of transient regime. Spatial discretization is performed by the Spectral Element method, and three different schemes are compared for time discretization. Péron \cite{victor-peron} has presented equivalent conditions and asymptotic models for the diffraction problem of elastic and acoustic waves in a solid medium surrounded by a thin layer of fluid medium in harmonic regime. Other noteworthy references in this field are \cite{popa,benthien-schenck,flemisch-kaltenbacher,hsiao-nigam,hsiao-sayas-weinacht,tonatiuh1,jeans-mathews,tromp,koreans}.

At the best of our knowledge, in all of the aforementioned works a well-posedness result for the mathematical formulation of the coupled problem cannot be found. In this work, the proof of existence and uniqueness for a strong solution is accomplished in a semigroup framework, by resorting to the Hille--Yosida theorem \cite[Chap.~7]{brezis}. Notice that a similar abstract setting wherein semigroup theory on Hilbert spaces can be invoked, was employed in \cite{tonatiuh2}; here, the problem of acoustic waves scattered by a piezoelectric solid is investigated.

The rest of the paper is organized as follows. In Section~\ref{sec:problem.statement} we give the formulation of the problem and prove the existence and uniqueness of the solution under suitable hypotheses on source terms and initial values. In Section~\ref{sec:disc.set} we introduce the discrete setting, with particular reference to the assumptions on the underlying polytopic mesh. In Section~\ref{sec:semi.disc.problem} we present the formulation of the {semi-discrete} problem.
In Section~\ref{sec:stability} we prove the stability of the semi-discrete formulation in a suitable energy norm. In Section~\ref{sec:error.estimates}, we prove $hp$-convergence results (with $h$ and $p$ denoting, as usual, the meshsize and the polynomial degree, respectively) for the error in the energy norm. The fully discrete formulation is discussed in Section~\ref{sec:fully.discrete}. Finally, in Section~\ref{sec:numerical.examples}, we present some numerical experiments carried out in a two-dimensional setting to validate the theoretical results. The proofs of two technical lemmas are postponed to \ref{app}.

%\section*{General notation}
%\begin{figure}%[h!]
%\includegraphics[scale=.35]{disegno}
%\caption{Schematic representation of the spatial domain where the problem is formulated.}
%\label{figura}
%\end{figure}
%
\section{The elasto-acoustic problem}\label{sec:problem.statement}
In what follows, scalar fields are represented by lightface letters, vector fields by boldface roman letters, and second-order tensor fields by boldface greek letters.
We let $\Omega \subset \bb R^d$, $d\in\{2,3\}$, denote an open bounded convex domain
with Lipschitz boundary, given by the union of two open disjoint bounded convex subdomains $\Omega_e$ and $\Omega_a$ representing an elastic and an acoustic domain, respectively. We denote by $\Gamma_{\mr I}
= \de\Omega_e \cap \de\Omega_a$ the \emph{interface}
between the two domains, also of Lipschitz regularity and with strictly positive surface measure. We assume that the following partitions hold:
$\de\Omega_e  = \Gamma_{eD} \cup \Gamma_{\mr I}$ and %\\
$\de\Omega_a  = \Gamma_{aD} \cup \Gamma_{\mr I}$,
where $\Gamma_{eD}$ and $\Gamma_{aD}$ also have strictly positive surface measure, and $\Gamma_{eD} \cap \Gamma_{\mr I} = \emptyset = \Gamma_{aD} \cap \Gamma_{\mr I}$. We further denote by $\bf n_e$ and $\bf n_a$ the outer unit normal vectors to $\de\Omega_e$ and $\de\Omega_a$,
respectively; thereby, $\bf n_e = -\bf n_a$ on $\Gamma_{\mr I}$. For $X\subseteq \wl\Omega$, we write $\bf L^2(X)$ in place of $L^2(X)^d$, with scalar product denoted by $(\dotp,\dotp)_X$ and associated norm $\|{\cdot}\|_X$. Analogously, we write $\bf H^l(X)$ in place of $H^l(X)^d$ for Hilbertian Sobolev spaces of vector-valued functions with index $l \ge 0$, equipped with norm $\|{\cdot}\|_{l,X}$ (so that $\|{\cdot}\|_{0,X} \equiv \|{\cdot}\|_X$ on $\bf H^0(X) \equiv \bf L^2(X)$). Given an integer $p\ge 1$, $\mathscr P_p(X)$ denotes the space spanned by polynomials of total degree at most $p$ on $X$. Given a subdivision $\cc T_h$ of $\Omega$
into disjoint open elements $\k$ such that $\wl \Omega = \bigcup_{\k\in\cc T_h} \wl\k$, we denote by $$\mathscr P_{\bf p}(\cc T_h) = \{v \in L^2(\Omega): v_{| \k} \in \mathscr P_{p_\k}(\k) \ \forall \k \in \cc T_h\}$$ the space of piecewise polynomial functions on $\cc T_h$, with $\bf p = (p_\k)_{\k\in \cc T_h}$, $p_\k \ge 1\, \forall \k\in\cc T_h$. Finally, for $T > 0$, we let $(0,T]$ denote a time interval. For the sake of readibility we omit, at times, the dependence on time $t \in (0,T]$. The first and second time derivatives of a scalar- or vector-valued function $\Psi = \Psi(t)$ are denoted by $\dot \Psi$ and $\ddot \Psi$, respectively.

The elasto-acoustic problem is formulated as follows:~for \emph{sufficiently smooth} loads per unit volume ${\bf f_e}$ and $f_a$, and initial conditions $(\bf u_0, \bf u_1, \phi_0,\phi_1)$, find $(\bf u,\phi)$ such that:
\begin{align}\label{strong_form}
\left\{
\begin{alignedat}{2}
\rho_e \ddot{\bf u} + 2\rho_e \z\dot{\bf u}  +
\rho_e\z^2 \bf u - \divb\bm\sigma(\bf u) &= {\bf f_e} &\qquad&\hbox{in } \Omega_e \times(0,T],\\[4pt]
\bf u &= \bf 0 &\qquad& \hbox{on }\Gamma_{eD} \times (0,T],\\
\bm\sigma(\bf u)\bf n_e & = - \rho_a\dot\phi\, \bf n_e &\qquad& \hbox{on }\Gamma_{\mr I}\times (0,T], 
\\[4pt]
\bf u(0) &= \bf u_0, &\qquad& \hbox{in }\Omega_e,\\
\dot{\bf u}(0) & = \bf u_1 &\qquad& \hbox{in }\Omega_e;\\[12pt]
c^{-2}\ddot\phi - \Del\phi &= f_a &\qquad&\hbox{in } \Omega_a \times (0,T],\\[4pt]
\phi &= 0 &\qquad& \hbox{on }\Gamma_{aD} \times (0,T],\\
{\de\phi}/{\de\bf n_a} &= - \dot{\bf u} \dotp \bf n_a &\qquad& \hbox{on }\Gamma_{\mr I}\times (0,T]; 
 \\[4pt]
\phi(0) &= \phi_0, &\qquad& \hbox{in }\Omega_a,\\
\dot{\phi}(0) & = \phi_1 &\qquad& \hbox{in }\Omega_a.
\end{alignedat}
\right.
\end{align}
Here, $\bf u \colon \wl\Omega_e \times [0,T] \to \bb R^3$ and $\phi\colon\wl \Omega_a \times [0,T] \to \bb R$ represent the {displacement vector} and the {acoustic potential}, respectively. Moreover, $\rho_e$ is the {density of the elastic body} $\Oe$, with $0 < \rho_e^- \le \rho_e \le \rho_e^+ < +\infty$ a.e.~in $\Oe$, $\bm\sigma(\bf u) = \bb C \bm\eps(\bf u)$ is the {Cauchy stress} tensor, with $\bb C$ the fourth-order, symmetric and uniformly elliptic elasticity tensor, and \mbox{$\bm\eps(\bf u) = \frac{1}{2}\left(\grad \bf u + \grad \bf u^T\right)$} is the {strain} tensor. Also, we denote by $\rho_a$ the {density of the acoustic region} $\Omega_a$, with $0 < \rho_a^- \le \rho_a \le \rho_a^+ < +\infty$ a.e.~in $\Oa$, and by $c > 0$ the {speed of the acoustic wave}. The \emph{damping factor} $\z \ge 0$, $\z\in L^\infty(\Omega)$, is a decay factor with the dimension of the inverse of time. Usually, in engineering applications, the viscoelastic behavior of a material is expressed through the adimensional quality factor $Q = Q_0 f/f_0$, where $f_0$ is a reference frequency and $Q_0 = \pi f_0 / \z$ \cite{kosloff}.

Notice that the coupled nature of the problem is to be ascribed to the trasmission conditions imposed on $\Gamma_{\mr I} \times (0,T]$.  The first one takes account of the acoustic pressure exterted by the fluid onto the elastic body through the interface, whereas the second one expresses the continuity of the normal component of the velocity field at the interface.

Let us now introduce the Hilbertian Sobolev spaces 
\begin{equation}\label{spaces}
\begin{alignedat}{2}
\bf H^1_D(\Omega_e) & = \{\bf v \in \bf H^1(\Omega_e) : \bf v = \bf 0\hbox{ on }\Gamma_{eD}\}, &\ \ %\\
%\quad
 H^1_D(\Omega_a) & = \{v \in H^1(\Omega_a) :  v =  0\hbox{ on }\Gamma_{aD}\}, \\[5pt]
\bf H^{\Del}_{\bb C}(\Omega_e) & = \{ \bf v \in \bf L^2(\Omega_e) : \divb\bb C\bm\eps(\bf v) \in \bf L^2(\Omega_e)\}, &\ \
 H^{\Del}(\Omega_a) & = \{ v \in L^2(\Omega_a) : \Del v \in L^2(\Omega_a)\}.
\end{alignedat}
\end{equation}
The existence and uniqueness of a strong solution to \eqref{strong_form} can be inferred in the framework of the Hille--Yosida theory. In particular, the following theorem holds.
\begin{theorem}[Existence and uniqueness]\label{well_posedness}
Assume that the initial data have the following regularity:
\begin{align}\label{initial_data}
%\begin{equation}\begin{array}{cc}
%\begin{aligned}
\bf u_0 \in  \bf H^\Del_{\bb C}(\Omega_e) \cap \bf H^1_D(\Omega_e),\ \ \bf u_1 \in \bf H^1_D(\Omega_e), \quad
\phi_0 \in H^\Del(\Omega_a) \cap H^1_D(\Omega_a), \ \ \phi_1 \in H^1_D(\Omega_a)
%\end{aligned}
%&\ \
%\begin{aligned}
%\end{aligned}
\end{align}
%\end{array}\end{equation}
and that the source terms are such that
\begin{equation}\label{source_term}
{\bf f_e} \in C^1([0,T];\bf L^2(\Omega_e)),\quad f_a \in C^1([0,T]; L^2(\Oa)).
\end{equation}
Then, problem \eqref{strong_form} admits a unique strong solution $(\bf u, \phi)$ such that
\begin{equation}\label{sol_regularity}
\begin{aligned}
\bf u &\in C^2([0,T];\bf L^2(\Omega_e)) \cap C^1([0,T];\bf H^1_D(\Omega_e))  \cap C^0([0,T];\bf H^\Del_{\bb C}(\Oe)\cap \bf H^1_D(\Oe)),\\[2ex]
\phi &\in C^2([0,T];L^2(\Omega_a))\cap C^1([0,T]; H^1_D(\Omega_a))  \cap C^0([0,T]; H^\Del(\Oa)\cap  H^1_D(\Oa)).
\end{aligned}
\end{equation}
\end{theorem}
\begin{remark}[Boundary conditions] We consider formulation \eqref{strong_form} for ease of presentation, but more general boundary conditions, such as Dirichlet and Neumann nonhomogeneous conditions, can be taken into account, provided the data are sufficiently regular. In this case, suitable trace liftings of boundary data have to be introduced, by resorting to a one-parameter family of static problems (where the parameter is time). Then, it can be shown that a result analogous to \eqref{sol_regularity} holds, provided boundary data have $C^3$-regularity in time (see, e.g., \cite[Theorem~1.1]{bonaldi-m3as}, where a similar issue arises).
\end{remark}
\begin{remark}[Convexity]
The above result also holds without any convexity assumption on neither $\Omega$ nor the subdomains $\Oe$ and $\Oa$. On the other hand, this hypothesis is necessary to ensure that the exact solution $(\bf u, \phi)$ is (at least) $H^2$-regular, so that the traces of  $\grad \bf u$ and $\grad\phi$ on  $(d-1)$-dimensional simplices are both well defined, in view of the forthcoming analysis of the semi-discrete problem (cf.~Section~\ref{sec:semi.disc.problem}).
\end{remark}
\begin{proof}[Proof of Theorem~\ref{well_posedness}] %See Section~\ref{app:well_posedness} of Appendix \ref{app}.
Let $\bf w = \dot{\bf u}$, $\pphi = \dot\phi$, and $\cc U= (\bf u, \bf w, \phi, \pphi)$. 
We introduce the Hilbert space
$$ \bb H = \bf H^1_D(\Oe) \times \bf L^2(\Oe) \times H^1_D(\Oa) \times L^2(\Oa),$$
equipped with the following scalar product:
\begin{equation}\label{sc_prod}
\begin{aligned}(\cc U_1,\cc U_2)_{\bb H} = & \ (\rho_e\z^2 \bf u_1,\bf u_2)_\Oe + \left(\bb C\bm\eps(\bf u_1),\bm\eps(\bf u_2)\right)_\Oe \\
& +(\rho_e \bf w_1,\bf w_2)_\Oe + (\rho_a \grad\phi_1,\grad\phi_2)_\Oa + (c^{-2}\rho_a \pphi_1,\pphi_2)_\Oa.
\end{aligned}
\end{equation}
Then, we define the operator $A \colon D(A) \subset \bb H \to \bb H$ by 
$$A \cc U = \big({-\bf w},\ { 2\z \bf w + \z^2 \bf u - \rho_e^{-1} \divb\bb C\bm\eps (\bf u)}, \ {-\pphi}, \ {-c^2\Del\phi}\big)\quad \forall \cc U \in D(A),$$
where the domain $D(A)$ of the operator is the linear subspace of $\bb H$ defined as follows (cf.~definition~\eqref{spaces}):
\begin{equation}\label{domain}
\begin{aligned}
D(A) = \Big\{  \cc U\in \bb H : & \ \bf u \in \bf H^\Del_{\bb C}(\Oe),\ \bf w \in \bf H^1_D(\Oe), % \\
 %\div \bb C\bm\eps(\bf u) \in \bf L^2(\Oe), \ 
%\bm\eps(\bf w) \in \bf L^2(\Oe),\\ & \ \Del\phi \in L^2(\Oa), \ \grad\pphi \in \bf L^2(\Oa); \\
%& 
 \ \phi \in H^\Del(\Oa),\ \pphi \in H^1_D(\Oa); \\
%&\ \bf u = \bf 0,\ \bf w = \bf 0\text{ on }\Gamma_{eD}; \ \phi = 0,\ \pphi = 0\text{ on }\Gamma_{aD}; \\
&\left(\bb C\bm\eps(\bf u) + \rho_a \pphi \bf I\right) \bf n_e = \bf 0\text{ on }\Gamma_{\mr I},
\ \ {\left(\grad\phi + \bf w\right)}\dotp{\bf n_a} = 0\text{ on }\Gamma_{\mr I} \Big\}.
\end{aligned}
\end{equation}
Finally, let $$\cc F = (\bf 0, \rho_e^{-1} {\bf f_e}, 0, c^2 f_a).$$
Problem \eqref{strong_form} can then be reformulated as follows:~given $\cc F \in C^1([0,T]; \bb H)$ and $\cc U_0 \in D(A)$, find \mbox{$\cc U\in C^1([0,T];\bb H) \cap C^0([0,T];D(A))$} such that
\begin{equation}\label{abstract_eq}
\begin{aligned}
\frac{\d\cc U}{\d t}(t) + A \cc U(t) &= \cc F(t),\quad t\in(0,T],\\
\cc U(0) &= \cc U_0.
\end{aligned}
\end{equation}
Owing to the Hille--Yosida Theorem, this problem is well-posed provided $A$ is maximal monotone, i.e., $(A\cc U,\cc U)_{\bb H} \ge 0\,\,\forall \cc U \in D(A)$ and $I+A$ is surjective from $D(A)$ onto $\bb H$. By the definition \eqref{sc_prod} of the scalar product in $\bb H$, we have
$$
\begin{aligned}
(A \cc U, \cc U)_{\bb H} = &\ (-\rho_e\z^2 \bf w,\bf u)_\Oe + (-\bb C\bm\eps(\bf w), \bm\eps(\bf u))_\Oe %\\ & \ 
+ \left(2\rho_e\z \bf w + \rho_e\z^2 \bf u  - \divb\bb C\bm\eps(\bf u), \bf w\right)_\Oe \\
& \
+ (-\rho_a\grad\varphi,\grad\phi)_\Oa + (-\rho_a\Del\phi, \varphi)_\Oa. 
\end{aligned}
$$
Taking into account the definition \eqref{domain} of the domain $D(A)$ and integrating by parts, we obtain
$$(A\cc U, \cc U)_{\bb H} = (2\rho_e\z\bf w,\bf w)_\Oe \ge 0,$$
i.e., $A$ is monotone. It then remains to verify that, for any $\cc F \equiv ({\bf F}_1, {\bf F}_2, F_3, F_4) \in \bb H$, there is (a unique) $\cc U \in D(A)$ such that $\cc U + A\cc U = \cc F$, that is,
\begin{equation}\label{maximality}
\hspace{-1cm}
\begin{aligned}
\bf u-\bf w &= {\bf F}_1,\\
(1+2\z)\bf w + \z^2 \bf u - \rho_e^{-1}\divb\bb C\bm\eps(\bf u) &= {\bf F}_2,\\
\phi -\pphi &= F_3,\\
\pphi -c^2\Del\phi & = F_4.
\end{aligned}
\end{equation}
The first and third equations allow to express $\bf w$ and $\pphi$ in terms of $\bf u$ and $\phi$, respectively; substituting these two relations in the other two equations gives
\begin{equation}\label{reduced_system}
\begin{aligned}
(\z+1)^2\bf u - \rho_e^{-1}\divb\bb C\bm\eps(\bf u) &=  (1+2\z){\bf F}_1 + {\bf F}_2,\\
\phi - c^2\Del\phi &= F_3 + F_4.
\end{aligned}
\end{equation}
Since $\bf n_e = -\bf n_a$ on $\Gamma_{\mr I}$, and owing to the first and third equations of \eqref{maximality} and to the transmission conditions on $\Gamma_{\mr I}$ embedded in the definition of $D(A)$, the variational formulation of the above problem reads:~find ${(\bf u, \phi)}\in{\bf H^1_D(\Oe)}\times{ H^1_D(\Oa)}$ such that, for any $(\bf v, \psi)\in \bf H^1_D(\Oe) \times H^1_D(\Oa)$,
$$\mathscr A((\bf u,\phi),(\bf v,\psi)) = \mathscr L(\bf v, \psi),$$
where
\begin{align*}
\mathscr A((\bf u, \phi),(\bf v,\psi)) = & 
\ (\rho_e(\z+1)^2 \bf u, \bf v)_\Oe + (\bb C\bm\eps(\bf u),\bm\eps(\bf v))_\Oe
+ (\rho_a\phi\bf n_e,\bf v)_{\Gamma_{\mr I}}  \\
& 
+ (\rho_a c^{-2} \phi,\psi)_\Oa + (\rho_a\grad\phi,\grad\psi)_\Oa - (\rho_a\bf u\dotp\bf n_e,\psi)_{\Gamma_{\mr I}}
\end{align*}
and
\begin{align*}
\mathscr L(\bf v,\psi) \! = %& \, 
(\rho_e(1+2\z)\bf F_1 + \rho_e \bf F_2,\bf v)_\Oe + (\rho_a F_3\bf n_e, \bf v)_{\Gamma_{\mr I}} %\\
%& 
+ (\rho_a c^{-2}(F_3+F_4),\psi)_\Oa - (\rho_a\bf F_1\dotp \bf n_e,\psi)_{\Gamma_{\mr I}}.
\end{align*}
This problem is well-posed owing to the Lax--Milgram Lemma (notice, in particular, that the bilinear form $\mathscr A$ is coercive since the interface contributions vanish when $\bf v = \bf u$ and $\psi=\phi$). In addition, thanks to equations \eqref{reduced_system} we infer that $\bf u \in \bf H^\Del_\bb C(\Oe) \cap \bf H^1_D(\Oe)$ and $\phi \in H^\Del(\Oa) \cap H^1_D(\Oa)$. This in turn gives {$(\bf w,\pphi) \in \bf H^1_D(\Oe)\times H^1_D(\Oa)$} thanks to the first and third equations of \eqref{maximality}. Thus, $\cc U\in D(A)$ and the proof is complete.
\end{proof}

With a view towards introducing the semi-discrete counterpart of \eqref{strong_form} and to carry out its analysis, we observe that the solution given by \eqref{sol_regularity} satisfies the following \emph{weak form} of \eqref{strong_form}: 
%for any $t\in (0,T]$, given ${\bf f_e}(t) \in \bf L^2(\Omega_e)$ and initial conditions $(\bf u_0, \bf u_1, \phi_0,\phi_1) \in \bf H^1_D(\Omega_e)\times \bf H^1(\Omega_e) \times H^1_D(\Omega_a)\times H^1(\Oa)$, g {$(\bf u(t),\phi(t)) \in \bf H^1_D(\Oe)\times H^1_D(\Oa)$} such that, 
for any $t\in (0,T]$, and all $(\bf v, \psi) \in \bf H^1_D(\Oe) \times H^1_D(\Oa)$,
\begin{equation}
\begin{multlined}\label{weak_form}
%\begin{alignedat}{1}
%\hspace{-1.5cm}
(\rho_e\ddot{\bf u}(t),\bf v)_{\Omega_e} +(c^{-2}\rho_a\ddot\phi(t),\psi)_{\Omega_a} + (2\rho_e\z \dot{\bf u}(t),\bf v)_{\Omega_e} + (\rho_e\z^2\bf u(t),\bf v)_{\Omega_e} \\[4pt]
+\cc A_e(\bf u(t),\bf v)  + \cc A_a(\phi(t),\psi) +\cc I_e(\dot\phi(t), \bf v) + \cc I_a(\dot{\bf u}(t), \psi)
\\  = ({\bf f_e}(t), \bf v)_{\Omega_e} + (\rho_a f_a(t),\psi)_\Oa.
\end{multlined}
\end{equation}
\begin{remark}[Weak formulation]
This does not guarantee, of course, that looking for a solution of weaker regularity is a well-posed problem. Nevertheless, in view of the semi-discrete energy error analysis (cf.~Section~\ref{sec:error.estimates}, in particular Theorem~\ref{error-estimate}), we have to assume that the solution is smooth enough in time and space.
\end{remark}
Here, the bilinear forms $\cc A_e \colon \bf H^1_D(\Omega_e) \times\bf H^1_D(\Omega_e) \to \bb R$, \, $\cc I_e \colon H^1_D(\Omega_a) \times \bf H^1_D(\Omega_e) \to \bb R$, {$\cc A_a \colon H^1_D(\Omega_a) \times H^1_D(\Omega_a) \to \bb R$}, and $\cc I_a\colon \bf H^1_D(\Omega_e) \times H^1_D(\Omega_a) \to \bb R$
are defined as follows:
\begin{align}
\begin{aligned}
\cc A_e(\bf u, \bf v) & = ({\bb C}\bm\eps(\bf u),\bm\eps(\bf v))_{\Omega_e}, \\
\cc A_a(\phi,\psi) & = (\rho_a\grad\phi,\grad\psi)_{\Omega_a},
\end{aligned}
&\qquad
\begin{aligned} 
{\cc I_e(\psi,\bf v)}  & = (\rho_a\psi \bf n_e, \bf v)_{\Gamma_{\mr I}}, \\
{\cc I_a(\bf v, \psi)} & = (\rho_a {\bf v}\dotp \bf n_a,\psi)_{\Gamma_{\mr I}}.
\end{aligned}
\end{align}
Notice that we have multiplied the second evolution equation by $\rho_a$ to ensure (skew) symmetry of the two interface terms (since $\bf n_a = -\bf n_e$).

%
%\begin{remark}[Terms with time derivatives]
%In \eqref{weak_form}, we are implicitly assuming that, for any \emph{fixed} $t \in (0,T]$, it holds $\ddot{\bf u}(t) \in \bf L^2(\Omega_e)$, $\dot{\bf u}(t) \in \bf H(\div;\Omega_e)$, $\ddot\phi(t) \in L^2(\Omega_a)$ and $\dot\phi(t) \in L^2(\Gamma_{\mr I})$. Of course, in order to give \eqref{weak_form} a meaning \emph{for all} $t \in (0,T]$, hence considering an external load \mbox{${\bf f_e} \in L^2(0,T;\bf L^2(\Omega_e))$}, the $L^2$-scalar products and bilinear forms containing time derivatives have actually to be replaced by suitable \emph{duality pairings}.
%\end{remark}
%
%
%------------------------------------------------------------------------------%
%------------------------------------------------------------------------------%

\section{Discrete setting}\label{sec:disc.set}
Assuming that $\Omega_e$ and $\Omega_a$ are polygonal or polyhedral, we now introduce a polytopic mesh $\cc T_h$ over $\Omega$. We denote by $h_\k$ the diameter of an element $\k\in\cc T_h$, and set $h=\max_{\k \in\cc T_h} h_\k$. We assume that $\cc T_h$ is \emph{compliant with the underlying geometry}, i.e., the decomposition $\cc T_h = \cc T_h^e \cup \cc T_h^a$ holds, where $\cc T_h^e = \{\k \in \cc T_h : \wl \k \subseteq \wl\Omega_e\}$ and $\cc T_h^a= \{\k \in \cc T_h : \wl \k \subseteq \wl\Omega_a\}$.
%, so that $\wl\Omega_e = \bigcup_{\k\in\cc T_h^e} \wl\k$ and $\wl\Omega_a = \bigcup_{\k\in\cc T_h^a} \wl\k$. 
We assume that $\bb C$ and $\rho_a$ are element-wise constant, and set 
%\begin{subequations}
%\begin{alignat}{2}
\begin{equation}\label{C_k}
\wl{\bb C}_\k =  ({ |\bb C^{\nf12}|_2^2})_{| \k} \ \ \forall\k \in \cc T_h^e, \qquad
 \wl{\rho}_{a,\k}  = {\rho_a}_{| \k} \ \ \forall \k \in \cc T_h^a.
%\end{alignat}
%\end{subequations}
\end{equation}
Here we have denoted by $|{\cdot}|_2$ the operator norm induced by the $\ell^2$-norm on $\bb R^n$, with $n$ the dimension of the space of symmetric second-order tensors ($n=3$ if $d=2$, $n=6$ if $d=3$). With each element of $\cc T_h^e$ (resp.~$\cc T_h^a$), we associate a polynomial degree $p_{e,\k} \ge 1$ (resp.~$p_{a,\k} \ge 1$), and introduce the following finite-dimensional spaces:
$$
\begin{alignedat}{1}
\bf V_{\! h}^e & = [\mathscr P_{\bf p_e}(\cc T_h^e)]^d = \left\{ {\bf v}_{h} \in \bf L^2(\Omega_e) : \bf v_{h | \k} \in [\mathscr P_{p_{e,\k}}(\k)]^d \ \forall \k \in \cc T_h^e \right\}, \\
V_h^a & = \mathscr P_{\bf p_a}(\cc T_h^a) = \left\{ \psi_h \in L^2(\Omega_a) : \psi_{h | \k} \in \mathscr P_{p_{a,\k}}(\k) \ \forall \k \in \cc T_h^a \right\}.\\
\end{alignedat}
$$
%We allow $p_{e,\k}$ and $p_{a,\k}$ to vary element-wise; we will point this out in the notation when necessary, replacing $p_{e,\k}$ by $p_{e,\k}$ and $p_{a,\k}$ by $p_{a,\k}$.
For an integer $l\ge 1$, we also introduce the {broken Sobolev spaces}
\begin{equation}\label{h_ell_brise}
\begin{aligned}
\bf H^l(\cc T_h^e)&=\left\{\bf v\in\bf L^2(\Omega_e)\,:\, \bf v_{|\k}\in \bf H^l(\k) \ \forall \k\in\cc T_h^e\right\},\\
H^l(\cc T_h^a)&=\left\{\psi \in L^2(\Omega_a)\,:\, \psi_{|\k}\in H^l(\k)\ \forall \k\in\cc T_h^a\right\}.
\end{aligned}
\end{equation}
%equipped, unless noted otherwise, with the broken seminorms $\|{\cdot}\|_{\bf H^l(\cc T_h)}$ and $\|{\cdot}\|_{H^l(\cc T_h)}$
%defined by
%\begin{equation}\label{norme_brisee}
%\begin{alignedat}{2}
%&\forall \bf v \in \bf H^{l}(\cc T_h), &\qquad \|\bf v\|_{H^{l}(\cc T_h)} &= \left(\sum_{\k\in\cc T_h}\|\bf v\|_{\bf H^{l}(\k)}^2\right)^{\nicefrac12} \!\!, \\
%&\forall \psi \in H^{l}(\cc T_h), &\qquad \|\psi\|_{H^{l}(\cc T_h)} &= \left(\sum_{\k\in\cc T_h}\|\psi\|_{H^{l}(\k)}^2\right)^{\nicefrac12} \!\!.
%\end{alignedat}
%\end{equation}

Henceforth, we write $x \lesssim  y$ and $x \gtrsim  y$ in place of $x \le Cy$ and $x \ge Cy$ respectively, for $C > 0$ independent of the discretization parameters (polynomial degree and meshsize), as well as of the \emph{number of faces} of a mesh element, but possibly depending on material properties, such as $\bb C$, $\rho_e$, $c$, and $\rho_a$.
\subsection{Grid assumptions}
We term \emph{interface} of $\cc T_{h}$ the intersection of the boundaries of any two neighboring elements of $\cc T_h$. This definition allows for the treatment of situations where hanging nodes or edges are present. Therefore, for $d = 2$, an interface will always consist of a piecewise linear segment. On the other hand, for $d = 3$, an interface will be given by the union of general polygonal surfaces; we thereby assume that each planar section of a given interface may be subdivided into a set of co-planar triangles. We refer to such $(d - 1)$-dimensional simplices (line segments for $d=2$, triangles for $d=3$), whose union determines an interface of $\cc T_h$, as \emph{faces}. We denote by $\cc F_h$ the set of all faces of $\cc T_h$. Also, let
\begin{equation}\label{k_interf}
\cc T_{h,\mr I} = \{\k \in \cc T_h : \de \k \cap \Gamma_{\mr I} \neq \emptyset\}
\end{equation}
denote the set of elements sharing a part of their boundary with $\Gamma_{\mr I}$, and $\cc T_{h,\mr I}^e = \cc T_{h,\mr I} \cap \cc T_h^e$, $\cc T_{h,\mr I}^a = \cc T_{h,\mr I} \cap \cc T_h^a$. We then define the set of faces laying on $\Gamma_{\mr I}$ as follows:
\begin{equation}\label{f_interf}
\cc F_{h,\mr I} = \{ F \in \cc F_h : F \subset \de\k^e \cap \de \k^a,\ \k^e \in \cc T_{h,\mr I}^e,\  \k^a \in \cc T_{h,\mr I}^a\}
\end{equation}
(see Figure~\ref{notation}).
Hence, we assume the following decomposition: $\cc F_h = \cc F_h^e \cup \cc F_{h,\mr I} \cup \cc F_h^a$, where $\cc F_h^e$ and $\cc F_h^a$ collect, respectively, all faces of $\cc T_h^e$ and of $\cc T_h^a$ that \emph{do not lay} on $\Gamma_{\mr I}$. Further, $\cc F_h^e$ and $\cc F_h^a$ are decomposed as follows: $\cc F_h^e = \cc F_h^{e,\mr i} \cup \cc F_h^{e,\mr b}$, $\cc F_h^a = \cc F_h^{a,\mr i} \cup \cc F_h^{a,\mr b}$, where $\cc F_h^{e,\mr i}$ and $\cc F_h^{a,\mr i}$ collect the \emph{internal faces} of $\cc T_h^e$ and $\cc T_h^a$, respectively, and $\cc F_h^{e,\mr b}$, $\cc F_h^{a,\mr b}$ collect the \emph{boundary faces} of $\cc T_h^e$ and $\cc T_h^a$, respectively.
\begin{figure}
\centering
\includegraphics[keepaspectratio=true,scale=.6]{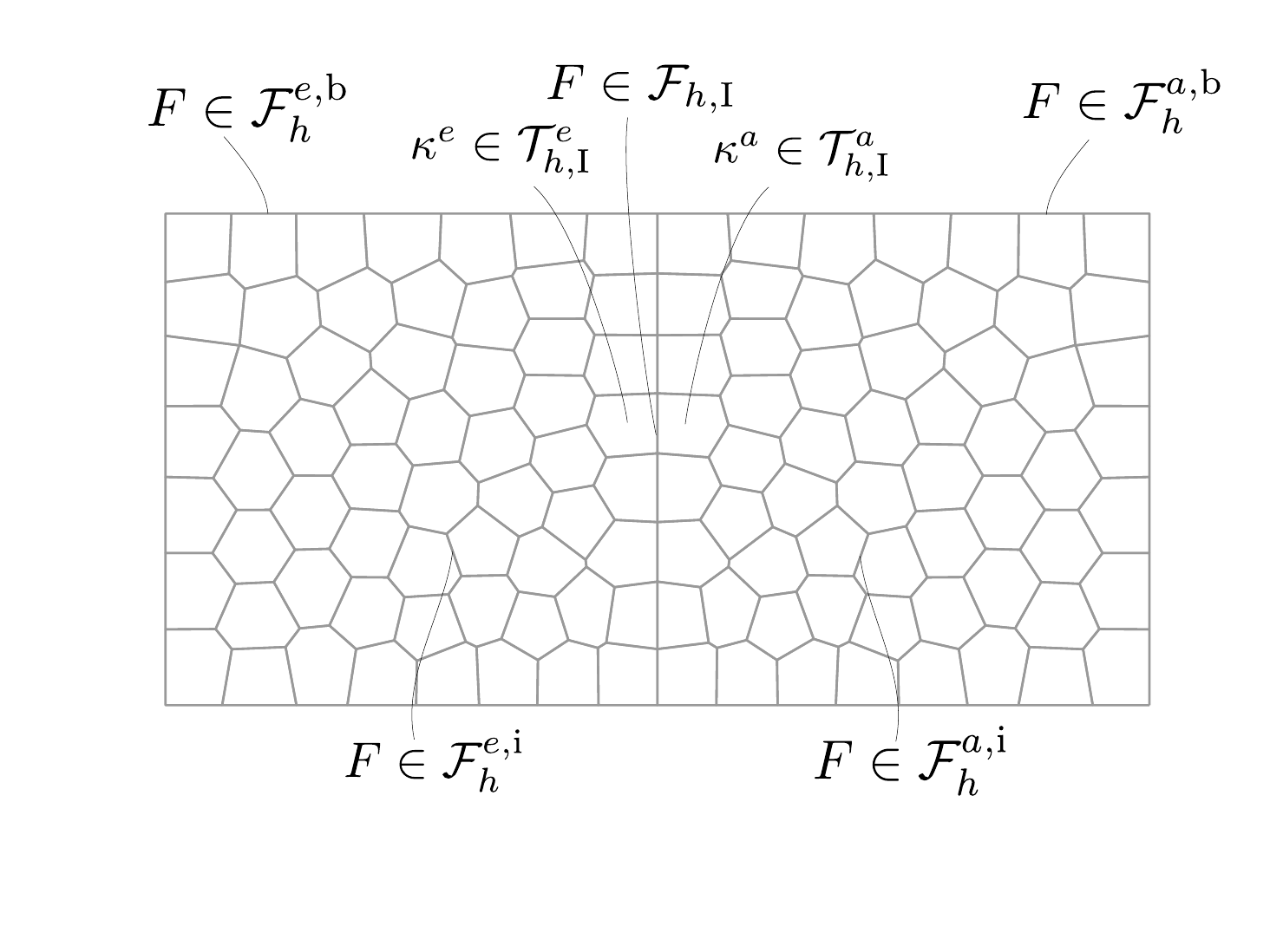}
\caption{Explanation of the employed notation for mesh elements and faces, in a two-dimensional setting. The left domain is elastic, the right one acoustic.}
\label{mesh}
\end{figure}

We can now proceed to state the main assumptions on $\cc T_{h}$, referring to \cite{cangiani-book} for further details.
\begin{subtheorem}{assumption}
\begin{assumption}
\label{uno}
Given an element $\k \in \cc T_{h}$, there exists a set of nonoverlapping (not necessarily shape-regular) $d$-dimensional simplices $\{\k_\flat^F\}_{F\subset \de\k} \subset \k$, such that, for any face $F\subset \de\k$,
$$
(\mr i) \ \displaystyle  h_\k \lesssim \frac{d |\k_\flat^F | }{|F|}, \qquad %\\[10pt]
(\mr{ii}) \ \displaystyle  \bigcup_{F\subset \de\k} \wl \k_\flat^F \subseteq \wl\k,
$$
where the hidden constant is independent of the discretization parameters, the number of faces of $\k$, the measure of $F$, and the material properties.
\end{assumption}
%\medskip
\begin{remark}[Number of faces and degenerating faces]
\label{remark:faces}
Notice that no restriction is imposed by Assumption \ref{uno} on either the number of faces of an element, or the measure of the face of an element with respect to the diameter of the element itself. Therefore, the case of faces degenerating under mesh refinement can be considered as well, cf.~also \cite{cangiani-book,antonietti-poligoni}. 
\end{remark}
We recall that, under Assumption \ref{uno}, the following \emph{trace-inverse inequality} holds for polytopic elements:
\begin{equation}\label{eq:trace.inv.ineq}
\forall \k \in \cc T_h,\,\forall v \in \mathscr P_p(\k),\quad \|v\|_{L^2(\de\k)} \lesssim p h_\k^{-\nf12} \|v\|_{L^2(\k)},
\end{equation}
where the hidden constant is independent of the discretization parameters, the number of faces per element, and the material properties \cite{cangiani-book}.
%
%\begin{assumption}\label{due}
%For any $\k \in \cc T_h$, we assume that $h_\k^d \ge |\k| \gtrsim h_\k^d$, with $d\in\{2,3\}$.
%\end{assumption}
%%
%\begin{assumption}\label{tre}
%Every polytopic element $\k \in \cc T_h$ admits a sub-triangulation into at most $m_\k \in \bb N$ nonoverlapping, shape-regular simplices $\mathfrak s_i$, $i \in \{1, \dots, m_\k\}$, such that $\wl \k = \cup_{i=1}^{m_\k} \ \wl{\mathfrak s}_i$ and $$|\mathfrak s_i| \gtrsim |\k|\quad\forall i \in \{1,\dots,m_\k\},$$
%where the hidden constant is indeadent of $\k$ and $\cc T_h$.
%\end{assumption}
%
\begin{assumption}
\label{due}
Let $\cc T_h^\sharp = \{\cc K\}$ denote a \emph{covering} of $\Omega = \Oe \cup \Oa$ consisting of shape-regular $d$-dimensional simplices $\cc K$. We assume that, for any $\k \in \cc T_h$, there exists $\cc K \in \cc T_h^\sharp$ such that $\k \subset \cc K$ and
$$\max_{\k \in \cc T_h}\card \left\{ \k^\prime \in \cc T_h : \k^\prime \cap \cc K \neq \emptyset, \ \cc K \in \cc T_h^\sharp\hbox{ such that } \k \subset \cc K \right\} \lesssim 1;$$
and that, for each pair $\k \in \cc T_h$, $\cc K \in \cc T_h^\sharp$ with $\k \subset \cc K$, $$\hbox{diam}(\cc K) \lesssim h_\k,$$
where the hidden constant is independent of the discretization parameters and of the material properties \cite{cangiani2014,cangiani-book}.
\end{assumption}
\begin{assumption}
\label{local_bv}
Let $\k^+,\k^-$ be any two neighboring elements of $\cc T_h$. We assume the following $hp$-\emph{local bounded variation property} for both the meshsize and the polynomial degree:
\begin{equation*}%\label{eq:local_bv}
 h_{\k^+} \lesssim h_{\k^-} \lesssim h_{\k^+}, \quad p_{\k^+} \lesssim p_{\k^-} \lesssim p_{\k^+},
\end{equation*}
where the hidden constant is independent of the discretization parameters, the number of faces per element, and the material properties \cite{perugia,mox-elastodynamics}.
\end{assumption}
\end{subtheorem}
%Also, under Assumptions \ref{due} and \ref{tre}, the following \emph{inverse estimate} holds:
%$$\forall \k \in \cc T_h,\,\forall v \in \mathscr P_p(\k),\quad \|\grad v\|_{L^2(\k)}^2 \lesssim {p^4}h_\k^{-2} \|v\|_{L^2(\k)}^2.$$
%
\section{Semi-discrete problem}\label{sec:semi.disc.problem}
Before stating the dG formulation of the semi-discrete problem we introduce the following average and jump operators \cite{unified-brezzi,arnold-RM}. For sufficiently smooth scalar-, vector- and tensor-valued fields $\psi$, $\bf v$, and $\bf\tau$, we define averages and jumps on an internal face $F \in \cc F_h^{e,\mr i} \cup \cc F_h^{a,\mr i}$, $F \subset  \de\k^+ \cap  \de\k^-$ with $\k^+$ and $\k^-$ any two neighboring elements in $\cc T_h^e$ or $\cc T_h^a$, as follows:
\begin{equation*}
\begin{alignedat}{2}
\jm{\psi} &= \psi^+\bf n^+ + \psi^- \bf n^-,&\qquad \av{\psi} &= \frac{\psi^+ + \psi^-}{2},\\
\jm{\bf v} & = \bf v^+ \otimes \bf n^+ + \bf v^- \otimes \bf n^-,&\qquad \av{\bf v} & = \frac{\bf v^+ + \bf v^-}{2},\\
\jm{\bf\tau} & = \bf\tau^+\bf n^+ + \bf\tau^-\bf n^-,&\qquad \av{\bf \tau} & = \frac{\bf \tau^+ + \bf\tau^-}{2},
\end{alignedat}
\end{equation*}
where $\bf a\otimes \bf b$ denotes the tensor product of $\bf a,\bf b \in \bb R^3$, and $\psi^\pm$, $\bf v^\pm$ and $\bf\tau^\pm$ denote the traces of $\psi$, $\bf v$ and $\bf \tau$ on $F$ taken from the interior of $\k^\pm$, and $\bf n^\pm$ is the outer unit normal vector to $\de\k^\pm$. When considering a {boundary face} $F \in \cc F_h^{e,\mr b} \cup \cc F_h^{a,\mr b}$, we set $\jm{\psi} = \psi\bf n$, $\jm{\bf v} = \bf v \otimes \bf n$, $\jm{\bf\tau} = \bf \tau\bf n$, and $\av{\psi} = \psi$, $\av{\bf v} = \bf v$, $\av{\bf \tau} = \bf\tau$. We also use the shorthand notation 
$$\la  \Phi, \Psi\ra_{\cc F} = \sum_{F\in\cc F} (\Phi, \Psi)_F,
\qquad 
\|\Phi\|_{\cc F} = \la \Phi, \Phi\ra_{\cc F}^{\nf12} $$ % = \left(\sum_{F\in\cc F} (\Phi, \Phi)_F \! \right)^{\nf12} $$ 
for scalar, vector or tensor fields $\Phi$ and $\Psi$ and for a generic collection of faces $\cc F \subset \cc F_{h}$.

The semi-discrete approximation of problem \eqref{weak_form} reads:
find $(\bf u_h, \phi_h)\! \in \! C^2([0,T];\! \bf V_{\! h}^e)  \times C^2([0,T]; \! V_h^a)$ such that, for all $(\bf v_h,\psi_h) \in \bf V_{\! h}^e \times V_h^a$,
\begin{equation}\label{coupled_dG}
\begin{multlined}
%\hspace{-.5cm}
(\rho_e\ddot{\bf u}_h(t),\bf v_h)_{\Omega_e} + (c^{-2}\rho_a\ddot\phi_h(t),\psi_h)_{\Omega_a}  %\\ 
+ (2\rho\z \dot{\bf u}_h(t),\bf v_h)_{\Omega_e} +
 (\rho\z^2\bf u_h(t),\bf v_h)_{\Omega_e}  \\[5pt]
% \hspace{.2cm}
 + \cc A_h^e(\bf u_h(t),\bf v_h)  + \cc A_h^a(\phi_h(t),\psi_h) %\\
+  \cc I_h^e(\dot\phi_h(t), \bf v_h) + \cc I_h^a(\dot{\bf u}_h(t), \psi_h) \\
  = ({\bf f_e}(t), \bf v_h)_{\Omega_e} + (\rho_a f_a(t),\psi_h)_\Oa,
\end{multlined}
\end{equation}
with initial conditions $\left({\bf u_h (0),\dot{\bf u}_h(0), \phi_h(0),\dot{\phi}_h(0)}\right) %=  \left ({\bf u_{0,h}, \bf u_{1,h},\phi_{0,h},\phi_{1,h}} \right) 
\in \bf V_{\! h}^e\times\bf V_{\! h}^e \times V_h^a\times V_h^a$, 
where the bilinear forms $\cc A_h \colon \bf V_{\! h}^e \times \bf V_{\! h}^e \to \bb R$, $\cc A_h^a \colon V_h^a \times V_h^a \to \bb R$,
$\cc I_h^e \colon V_h^a \times \bf V_{\! h}^e \to \bb R$ and $\cc I_h^a \colon \bf V_{\! h}^e \times V_h^a \to \bb R$ are given by
\begin{equation}\label{discr_bilin_forms}
\begin{alignedat}{2}
 \cc A_h^e(\bf u,\bf v) = & \ (\bm\sigma_h(\bf u ),\bm\eps_h(\bf v))_{\Omega_e} - \la \av{\bm\sigma_h(\bf u)}, \jm{\bf v} \ra_{{\cc F}_h^e} %\\  & 
 - \la \jm{\bf u }, \av{\bm\sigma_h(\bf v)}\ra_{{\cc F}_h^e} + \la \eta\jm{\bf u },\jm{\bf v}\ra_{{\cc F}_h^e} %&\ & \forall \bf u, \bf v \in \bf V_{\! h}^e,
 \\[5pt]
 \cc A_h^a(\phi,\psi)  = & \ (\rho_a \grad_h\phi,\grad_h\psi)_{\Omega_a} - \la \av{\rho_a \grad_h \phi},\jm{\psi} \ra_{{\cc F}_h^a} %\\ & 
 -  \la  \jm{\phi},\av{\rho_a\grad_h \psi} \ra_{{\cc F}_h^a} + \la\chi \jm{\phi},\jm{\psi} \ra_{{\cc F}_h^a} %&\ &\forall \phi,\psi\in V_h^a,
 \\[5pt]
 \cc I_h^e(\psi,\bf v) = & \ (\rho_a\psi\bf n_e,\bf v)_{\Gamma_\mr I}, %= &\ & \forall (\psi,\bf v) \in V_h^a \times \bf V_{\! h}^e, \\[5pt]
 \quad  \cc I_h^a(\bf v,\psi) = ({\rho_a \bf v}\dotp{\bf n_a},\psi )_{\Gamma_\mr I} = -\cc I_h^e(\psi,\bf v),
% &\quad& \forall (\bf v,\psi) \in \bf V_{\! h}^e \times V_h^a,
\end{alignedat}
\end{equation}
with $\grad_h$ the usual broken gradient operator on $\cc T_h$.
We point out that the last identity in \eqref{discr_bilin_forms} holds due to the fact that $\bf n_a = - \bf n_e$.
Here we have set, for any integer $l\ge 1$ and any $\bf v \in \bf H^l(\cc T_h^e) \supset \bf V_{\! h}^e$,
$$
\begin{aligned}
\bm\eps_h(\bf v) = \frac{1}{2}\left(\grad_h \bf v + \grad_h \bf v^T\right),\quad \bm\sigma_h(\bf v) = \bb C \bm\eps_h(\bf v).
\end{aligned}
$$
The \emph{stabilization functions} \mbox{$\eta \in L^\infty({\cc F}_h^e)$} and $\chi \in L^\infty({\cc F}_h^a)$ are defined as follows:
\begin{subequations}
%\begin{alignat}{3}
\begin{align}
\eta_{| F} & = \begin{cases}
\displaystyle
 \alpha \max_{\k \in \{\k^+,\k^-\}}
\left( \frac{\wl{\bb C}_\k p^2_{e,\k}}{h_{\k}} \right) &\quad\forall F \in {\cc F}_h^{e,\mr i},
\quad \ F \subseteq \de {\k^+} \cap \de{\k^-}, \\[15pt]
\displaystyle \frac{\wl{\bb C}_\k p^2_{e,\k}}{h_\k} &\quad\forall F \in {\cc F}_h^{e,\mr b},
\quad  F \subseteq \de \k;
\end{cases} \label{stabiliz.a} \\[10pt]
\chi_{| F} & = \begin{cases}
\displaystyle \beta \max_{\k\in \{\k^+,\k^-\}} \left(
\frac{\wl\rho_{a,\k} p_{a,\k}^2}{h_\k}\right) &\ \, \forall F \in {\cc F}_h^{a,\mr i},
\quad \, F \subseteq \de {\k^+} \cap \de{\k^-}, \\[15pt]
\displaystyle \frac{\wl\rho_{a,\k} p_{a,\k}^2}{h_\k} &\ \, \forall F \in {\cc F}_h^{a,\mr b},
\quad  F \subseteq \de \k.
\end{cases} \label{stabiliz.b}
\end{align}
\end{subequations}
where $\alpha,\beta > 0$ are positive constants to be properly chosen.
We now introduce the following norms:
\begin{equation}\label{norme}
\begin{alignedat}{2}
\|\bf v\|_{\dG,e}^2 & = \|\bb C^{\nf{1}{2}}\bm\eps_h(\bf v)\|_{\Omega_e}^2 + \|{\eta^{\nf{1}{2}}\jm{\bf v}}\|_{{\cc F}_h^e}^2
&\quad& \forall \bf v \in \bf H^1(\cc T_h^e) \supset \bf V_{\! h}^e,\\
\vvvert \bf v\vvvert_{\dG,e}^2 & = \|\bf v\|_\dGe^2 + \|{\eta^{-\nf{1}{2}}\av{\bb C\bm\eps_h(\bf v)}}\|_{{\cc F}_h^e}^2 &\quad& \forall \bf v \in \bf H^2(\cc T_h^e), \\
\|\psi\|_{\dG,a}^2 & = \|\rho_a^{\nf{1}{2}}\grad_h\psi\|_{\Oa}^2 + \|{\chi^{\nf{1}{2}} \jm{\psi}}\|_{{\cc F}_h^a}^2 &\quad& \forall \psi \in H^1(\cc T_h^a) \supset V_h^a,\\
\vvvert \psi\vvvert_{\dGa}^2 & = \|\psi\|_\dGa^2 + \|{\chi^{-\nf{1}{2}}\av{\rho_a \grad_h\psi}}\|_{{\cc F}_h^a}^2 &\quad& \forall \psi \in H^2(\cc T_h^a).
\end{alignedat}
\end{equation}
The following result follows based on employing standard arguments.
\begin{lemma}[Coercivity and boundedness of $\cc A_h^e$ and $\cc A_h^a$]
\label{coerc_bound}
Provided that $\cc T_h$ satisfies Assumption \ref{uno}, and that constants $\alpha$ and $\beta$ in \eqref{stabiliz.a}--\eqref{stabiliz.b} are chosen sufficiently large, the following continuity and coercivity bounds hold:
\begin{subequations}
\begin{equation}\label{u_cont_coerc}
\begin{aligned}
\cc A_h^e(\bf u, \bf v) & \lesssim \|\bf u\|_{\dG,e}\|\bf v\|_{\dG,e}&\quad&\forall \bf u,\bf v \in \bf V_{\! h}^e, \\
\cc A_h^e(\bf v,\bf v) & \gtrsim \|\bf v\|_{\dG,e}^2&\quad&\forall \bf v \in \bf V_{\! h}^e,
\end{aligned}
\end{equation}% \\[5pt]
\begin{equation}\label{phi_cont_coerc}
\begin{aligned}
\cc A_h^a(\phi, \psi) & \lesssim \|\phi\|_{\dG,a}\|\psi\|_{\dG,a}&\quad&\forall \phi,\psi \in V_h^a, 
 \\
\cc A_h^a(\psi,\psi) & \gtrsim \|\psi\|_{\dG,a}^2&\quad&\forall \psi \in  V_h^a.
\end{aligned}
\end{equation}
\end{subequations}
Moreover,
\begin{equation}\label{cont_estesa}
%\begin{alignedat}{3}
%&\qquad\qquad\quad \ \, &\cc A_h^e(\bf w, \bf v) &  \lesssim \tpn{\bf w}_{\dG,e}\|\bf v\|_{\dG,e}&\quad& \ \ \forall (\bf w,\bf v) \in \bf H^2(\cc T_h^e) \times \bf V_{\! h}^e, \\
%&\qquad\qquad\quad \ \, &\cc A_h^a(\phi, \psi)  & \lesssim \tpn{\phi}_{\dG,a}\|\psi\|_{\dG,a}&\quad&\ \ \forall (\phi,\psi) \in H^2(\cc T_h^a) \times V_h^a.
%\end{alignedat}
\begin{aligned}
\cc A_h^e(\bf w, \bf v) &  \lesssim \tpn{\bf w}_{\dG,e}\|\bf v\|_{\dG,e}&\quad&\forall (\bf w,\bf v) \in \bf H^2(\cc T_h^e) \times \bf V_{\! h}^e, \\
\cc A_h^a(\phi, \psi)  & \lesssim \tpn{\phi}_{\dG,a}\|\psi\|_{\dG,a}&\quad&\forall (\phi,\psi) \in H^2(\cc T_h^a) \times V_h^a.
\end{aligned}
\end{equation}
\end{lemma}
As a consequence of \eqref{u_cont_coerc}--\eqref{phi_cont_coerc}, whose proof hinges on Lemma \ref{lemma_stab}, the theory of ordinary differential equations guarantees that problem \eqref{coupled_dG} is well-posed (notice, also, that the coupling terms stemming from bilinear forms $\cc I_h^e$ and $\cc I_h^a$ do not contribute to the energy of the system, cf.~Remark~\ref{energy_norm_remark} below).
\section{Stability of the semi-discrete formulation}\label{sec:stability}
In this section we prove a stability result for the semi-discrete problem \eqref{coupled_dG} (see \cite{antonietti-mazzieri-polyhedra,mox-elastodynamics,ayuso,antonietti-ferroni-mazzieri-quarteroni} for the purely elastic case). Let $\bf W = (\bf v,\psi) \in C^1([0,T];\bf V_{\! h}^e) \times C^1([0,T];V_h^a)$;
we introduce the following mesh-dependent \emph{energy norm}
%\footnote{
%We employ here the term \emph{norm} and the notation $\|{\cdot}\|_\E$, $\|{\cdot}\|_{\E_e}$, $\|{\cdot}\|_{\E_a}$ with a slight abuse. The quantities defined by \eqref{energy_norm}--\eqref{energies} are not, indeed, norms on $C^2([0,T];\bf V_{\! h}^e) \times C^2([0,T];V_h^a)$, $C^2([0,T];\bf V_{\! h}^e)$, and $C^2([0,T];V_h^a)$, respectively. This is thus to be understood as a shorthand notation.
%}:
\begin{equation}\label{energy_norm}
\|\W(t)\|_\E^2 = \| \bf v(t)\|_{\E_e}^2 + \| \psi(t)\|_{\E_a}^2,
\end{equation}
where
\begin{equation}\label{energies}
\begin{aligned}
\| \bf v(t)\|_{\E_e}^2 &= 
%\begin{dcases}
\| \rho_e^{\nf{1}{2}}\dot{\bf v}(t)\|_\Oe^2 + \|\rho_e^{\nf{1}{2}}\z \bf v(t)\|_\Oe^2 + \|\bf v(t)\|_\dGe^2, 
%& t\in (0,T],\\
%\|\rho_e^{\nf{1}{2}}\bf u_{1,h}\|_\Oe^2 + \|\rho_e^{\nf{1}{2}}\z \bf u_{0,h}\|_\Oe^2 + \| \bf u_{0,h}\|_\dGe^2, & t = 0,
%\end{dcases}
\\[5pt] 
%\end{equation}
%and
%\begin{equation}
\| \psi(t)\|_{\E_a}^2 &=
%\begin{dcases}
\|c^{-1}\rho_a^{\nf{1}{2}}
\dot\psi(t)\|_\Oa^2 + \|\psi(t)\|_\dGa^2.
%, & t\in (0,T], \\
%\| c^{-1}\rho_a^{\nf{1}{2}} \phi_{1,H}\|_\Oa^2 + \|\phi_{0,H}\|_\dGa^2, & t = 0.
%\end{dcases}
\end{aligned}
\end{equation}
\begin{remark}[Energy norm]
\label{energy_norm_remark}
The definition of the energy norm does not take into account the interface terms. The reason is related to the fact that, as observed previously, 
the bilinear forms $\cc I_h^e$ and $\cc I_h^a$ are skew-symmetric, i.e., $\cc I_h^a(\bf v,\psi) = -\cc I_h^e(\psi,\bf v)$ for all $(\bf v, \psi) \in \bf V_{\! h}^e \times V_h^a$.
\end{remark}
\begin{theorem}[Stability of the semi-discrete formulation]\label{stability}
Let $\bf U_h = (\bf u_h,\phi_h)$ be the solution of \eqref{coupled_dG}. For sufficiently large penalty parameters $\alpha$ and $\beta$ in \eqref{stabiliz.a} and \eqref{stabiliz.b}, respectively, the following bound holds:
% provided ${\bf f_e} \in L^2(0,T;\bf L^2(\Oe))$,
\begin{equation}\label{eq:stability}
\|\U_h(t)\|_\E \lesssim \|\U_h(0)\|_\E + \int_0^t \left( \|{\bf f_e}(\tau)\|_\Oe + \|f_a(\tau)\|_\Oa \right)\,\d\tau,\quad t \in (0,T].
\end{equation}
\end{theorem}
\begin{proof}
Taking $\bf v_h = \dot{\bf u}_h$ and $\psi_h = \dot\phi_h$ in \eqref{coupled_dG}, we obtain
\begin{multline*}
(\rho_e\ddot{\bf u}_h,\dot{\bf u}_h)_\Oe + (2\rho_e\z\dot{\bf u}_h,\dot{\bf u}_h)_\Oe 
+ (\rho_e\z^2\bf u_h,\dot{\bf u}_h)_\Oe
+ \big(\bm\sigma_h(\bf u_h),\bm\eps_h(\dot{\bf u}_h)\big)_\Oe \\ - \la \av{\bm\sigma_h(\bf u_h)}, \jm{\dot{\bf u}_h} \ra_{{\cc F}_h^e} - \la \jm{\bf u_h}, \av{\bm\sigma_h(\dot{\bf u}_h)}\ra_{{\cc F}_h^e} + \la \eta\jm{\bf u_h},\jm{\dot{\bf u}_h}\ra_{{\cc F}_h^e} 
\\
+ (c^{-2}\rho_a\ddot\phi_h,\dot\phi_h)_{\Omega_a} + (\rho_a\grad_h\phi_h,\grad_h\dot\phi_h)_\Oa 
 - \la\av{\rho_a \grad_h \phi_h},\jm{\dot\phi_h}\ra_{{\cc F}_h^a} 
 \\
 - \la \rho_a \jm{\phi_h},\av{\grad_h\dot\phi_h}\ra_{{\cc F}_h^a} + \la \chi\jm{\phi_h},\jm{\dot\phi_h}\ra_{{\cc F}_h^a}
  = ({\bf f_e},\dot{\bf u}_h)_\Oe + (\rho_a f_a,\dot \phi_h)_\Oa,
\end{multline*}
that is,
\begin{multline*}
\frac{1}{2}
\frac{\d}{\d t}
\left(\|\U_h\|_\E^2 - 2\left( 
\la \av{\bm\sigma_h(\bf u_h)},\jm{\bf u_h}\ra_{{\cc F}_h^e} + \la\av{\rho_a \grad_h \phi_h},\jm{\phi_h}\ra_{{\cc F}_h^a}
\right)
\right) \\
+ 2\|\rho_e^{\nf{1}{2}}\z^{\nf{1}{2}}\dot{\bf u}_h\|_\Oe^2 = ({\bf f_e},\dot{\bf u}_h)_\Oe + (\rho_a {f_a},\dot{\phi}_h)_\Oa.
\end{multline*}
Integrating the above identity over the interval $(0,t)$ we have
\begin{multline*}
\hspace{-.3cm}
\|\U_h(t)\|_\E^2 - 2\left( \la \av{\bm\sigma_h(\bf u_h(t))},\jm{\bf u_h(t)}\ra_{{\cc F}_h^e} + \la\av{\rho_a\grad_h\phi_h(t)},\jm{\phi_h(t)}\ra_{{\cc F}_h^a}\right) \\
+ \,4 \! \int_0^t \|\rho_e^{\nf{1}{2}} \z^{\nf{1}{2}} \dot{\bf u}_h(\tau)\|_\Oe^2\d\tau 
=  \|\U_h(0)\|_\E^2 - 2\left( \la \av{\bm\sigma_h(\bf u_h(0))},\jm{\bf u_h(0)}\ra_{{\cc F}_h^e} \right. \\  + \left. \la\av{\rho_a\grad_h\phi_h(0)},\jm{\phi_h(0)}\ra_{{\cc F}_h^a}\right) %\\ 
+ 2 \int_0^t \! \big({\bf f_e}(\tau),\dot{\bf u}_h(\tau)\big)_\Oe\d\tau
+2 \int_0^t \! \big(\rho_a {f_a}(\tau),\dot{\phi}_h(\tau)\big)_\Oa\d\tau,
\end{multline*}
and, since the last term on the left-hand side is positive, we get
\begin{multline*}
\|\U_h(t)\|_\E^2 - 2\left( \la \av{\bm\sigma_h(\bf u_h(t))},\jm{\bf u_h(t)}\ra_{{\cc F}_h^e} + \la\av{\rho_a\grad_h\phi_h(t)},\jm{\phi_h(t)}\ra_{{\cc F}_h^a}\right)
\\
 \le \|\U_h(0)\|_\E^2 - 2\left( \la \av{\bm\sigma_h(\bf u_h(0))},\jm{\bf u_h(0)}\ra_{{\cc F}_h^e} + \la\av{\rho_a\grad_h\phi_h(0)},\jm{\phi_h(0)}\ra_{{\cc F}_h^a}\right) \\
+ 2 \int_0^t \! \big({\bf f_e}(\tau),\dot{\bf u}_h(\tau)\big)_\Oe\,\d\tau
+ 2 \int_0^t \! \big(\rho_a { f_a}(\tau),\dot{\phi}_h(\tau)\big)_\Oa\,\d\tau.
\end{multline*}
From Lemma \ref{bounds} in the Appendix, we get
\begin{equation*}
\begin{aligned}
\|\U_h(t)\|_\E^2 - 2\left( \la \av{\bm\sigma_h(\bf u_h(t))},\jm{\bf u_h(t)}\ra_{{\cc F}_h^e} + \la\av{\rho_a\grad_h\phi_h(t)},\jm{\phi_h(t)}\ra_{{\cc F}_h^a}\right) & \gtrsim \|\U_h(t)\|_\E^2,\\
\|\U_h(0)\|_\E^2 - 2\left( \la \av{\bm\sigma_h(\bf u_h(0))},\jm{\bf u_h(0)}\ra_{{\cc F}_h^e} 
+\,\la\av{\rho_a\grad_h\phi_h(0)},\jm{\phi_h(0)}\ra_{{\cc F}_h^a}\right) & \lesssim \|\U_h(0)\|_\E^2,
\end{aligned}
\end{equation*}
where the first bound holds if the stability parameters $\alpha$ and $\beta$ are chosen large enough. Consequently
$$
\begin{aligned}\|\U_h(t)\|_\E^2 & \lesssim \|\U_h(0)\|_\E^2 + 2 \int_0^t \big({\bf f_e}(\tau),\dot{\bf u}_h(\tau)\big)_\Oe\,\d\tau + 2 \int_0^t \big(\rho_a {f_a}(\tau),\dot{\phi}_h(\tau)\big)_\Oa\,\d\tau \\
&\lesssim \|\U_h(0)\|_\E^2  + \int_0^t \|{\bf f_e}(\tau)\|_\Oe \|\rho_e^{\nf{1}{2}}\dot{\bf u}_h(\tau)\|_\Oe 
 + \int_0^t \|{f_a}(\tau)\|_\Oa \|c^{-1}\rho_a^{\nf{1}{2}}\dot{\phi}_h(\tau)\|_\Oa  \\
&\lesssim \|\U_h(0)\|_\E^2 + 
\int_0^t \left( \|{\bf f_e}(\tau)\|_\Oe + \|f_a(\tau)\|_\Oa \right)\|\U_h(\tau)\|_\E\,\d\tau,
\end{aligned}
$$
where we have used the Cauchy--Schwarz inequality and the definition \eqref{energy_norm} of the energy norm in the last two bounds.
The assertion follows then by employing Gronwall's Lemma (see e.g.~\cite[p.~28]{quarteroni}).
\end{proof}
\section{Semi-discrete error estimate}\label{sec:error.estimates}
The main subject of this section is the derivation of an \emph{a priori} error estimate for the  semi-discrete coupled problem \eqref{coupled_dG}. 
%We first introduce the following \emph{augmented energy norms} (cf.~\eqref{energies}): for $(\bf u(t),\dot{\bf u}(t))\in \bf H^1(\cc T_h^e)\times \bf H^1(\cc T_h^e)$ and $(\phi(t),\dot\phi(t))\in H^1(\cc T_h^a) \times H^1(\cc T_h^a)$,
%\begin{subequations}
%\begin{align}
%\vvvert{\bf u(t)}\vvvert_{\E_e}^2 & =  \|\bf u(t)\|_{\E_e}^2 + \! \sum_{F\in\cc F_{h,\mr I}}
%\frac{h_\k}{p_{e,\k}^2}  \|\rho_a^{\nf12}\dot{\bf u}(t) \|_F^2, \label{augmented_U}
%\\
%\vvvert{\phi(t)}\vvvert_{\E_a}^2 & =  \|\phi(t)\|_{\E_a}^2 + \! \sum_{
%F\in\cc F_{h,\mr I}}
%\frac{h_\k}{p_{a,\k}^2} \|\rho_a^{\nf12}\dot{\phi}(t)\|_F^2. \label{augmented_Phi}
%\end{align}
%\end{subequations}
%%
%\begin{remark}[Equivalence of energy norms]
%\label{equivalenza_norme}
%Using the trace-inverse inequality \eqref{eq:trace.inv.ineq}, it is easy to see that the augmented energy norms $\vvvert{\cdot}\vvvert_{\E_e}$ and $\vvvert{\cdot}\vvvert_{\E_a}$ are \emph{equivalent} to the energy norms $\|{\cdot}\|_{\E_e}$ and $\|{\cdot}\|_{\E_a}$ on the discrete spaces $\bf V_{\! h}^e$ and $V_h$, respectively.
%\end{remark}

For an open bounded polytopic domain $D \subset \bb R^d$, and a generic polytopic mesh $\cc T_h$ over $D$, we introduce, for any $\k \in \cc T_h$ and $m \in \bb N_0$, the \emph{extension operator} ${\mathscr E}\colon H^m(\k) \to  H^m(\bb R^d)$ such that $\mathscr E v_{| \k} = v$, $\|\mathscr E v\|_{H^m(\bb R^d)} \le C\|v\|_{H^m(\k)}$, with $C>0$ depending only on $m$ and $\k$. The corresponding vector-valued version, mapping $\bf H^m(\k)$ onto $\bf H^m(\bb R^d)$, acts component-wise and will be denoted in the same way. 
%We recall the following result (see, e.g., \cite{cangiani-book}).
%%
%\begin{lemma}[Interpolation estimates I]
%\label{cangiani_approx} 
%Suppose that Assumptions \ref{uno} and \ref{due} hold for $\mf T_h$. Let {$v_{| \k} \in H^m(\k)$}, $m > 1 + d/2$, such that $\mathscr E  v_{| \cc K} \in H^m(\cc K)$ for any $\k \in \mf T_h$, with $\k \subset \cc K$, $\cc K\in \mf T_h^\sharp$. Then there exist a \emph{projection operator} ${\Pi}\colon L^2(D) \to \mathscr P_{\bf p}(\mf T_h)$ and a real number $C>0$ independent of the discretization parameters and the number of faces per element such that
%\begin{subequations}
%\begin{align}
%\| v - {\Pi} v\|_{r,\k} & \le C \frac{h_\k^{\min(m,p_\k+1)-r}}{p_\k^{m-r}}
%\| \mathscr E v\|_{m,\cc K} \quad\forall r \in \{0,\dots,m\}, \label{H^r-norm}\\
%%
% \| v - { \Pi} v\|_{\de\k} &\le C \frac{h_\k^{\min(m,p_\k+1)-\nf12}}{p_\k^{m-\nf12}}\| \mathscr E v\|_{m,\cc K}, \label{L^2-b-norm}\\
%%
%\|\grad( v - { \Pi} v)\|_{\de\k} & \le C \frac{h_\k^{\min(m,p_\k+1)-\nf32}}{p_\k^{m-\nf32}}\| \mathscr E v\|_{m,\cc K}. \label{H^1-b-norm}
%\end{align}
%\end{subequations}
%\end{lemma}
%%
%The vector-valued version $\bf\Pi$ of the projection operator, mapping $\bf L^2(D)$ onto $\mathscr P_\bf p(\mf T_h)^d$, acts component-wise. In the following, we denote $\bf v_I = {\bf\Pi}\bf v \in \bf V_{\! h}^e$, for $D=\Oe$ and $\mf T_h = \cc T_h^e$, and $\psi_I = \Pi \psi \in V_h^a$, for $D=\Oa$ and $\mf T_h = \cc T_h^a$.
The result below is a consequence of the $hp$-approximation properties stated in \cite[Lemmas 23 and 33]{cangiani-book} and of Assumption \ref{local_bv} on local bounded variation.
\begin{lemma}[Interpolation estimates]
\label{interpolation_estimates}
For any pair of functions $(\bf v,\psi) \in \bf H^m(\cc T_h^e) \times H^n(\cc T_h^a)$, $m\ge 2$, $n\ge 2$, there exists a pair of interpolants $(\bf v_I,\psi_I)\in \bf V_{\! h}^e\times V_h^a$ such that
$$
\begin{aligned}
\vvvert{\bf v - \bf v_I}\vvvert_\dGe^2 &\lesssim \sum_{\k\in\cc T_h^e} \frac{h_\k^{2\min(m,p_{e,\k}+1)-2}}{p_{e,\k}^{2m-3}}
\| \mathscr E\bf v\|^2_{m,\cc K}, \\
\vvvert{\psi - \psi_I}\vvvert_\dGa^2 &\lesssim \sum_{\k\in\cc T_h^a} \frac{h_\k^{2\min(n,p_{a,\k}+1)-2}}{p_{a,\k}^{2n-3}}\| \mathscr E\psi\|^2_{n,\cc K}.
\end{aligned}
$$
Additionally, if $({\bf v},\psi) \in C^1([0,T];\bf H^m(\cc T_h^e)) \times C^1([0,T];H^n(\cc T_h^a))$, $m \ge 2$, $n \ge 2$, then
$$
\begin{aligned}
\|\bf v - \bf v_I\|_{\E_e}^2  &\lesssim \sum_{\k\in\cc T_h^e} \! \frac{h_\k^{2\min(m,p_{e,\k}+1)-2}}{p_{e,\k}^{2m-3}}\left(
\|\mathscr E \dot{\bf v}\|_{m,\cc K}^2+\|\mathscr E \bf v\|_{m,\cc K}^2\right), \\
\|\psi - \psi_I\|_{\E_a}^2 &\lesssim \sum_{\k\in\cc T_h^a} \! \frac{h_\k^{2\min(n,p_{a,\k}+1)-2}}{p_{a,\k}^{2n-3}}\left(
\|\mathscr E \dot{\psi}\|_{n,\cc K}^2+\|\mathscr E \psi\|_{n,\cc K}^2\right).
\end{aligned}
$$
\end{lemma}
We are now ready to state the main result of this section.
\begin{theorem}[\emph{A priori} error estimate in the energy norm]
\label{error-estimate}
Let Assumptions \ref{uno}--\ref{local_bv} hold. Assume that the exact solution of problem \eqref{strong_form}
is such that $\bf u\in C^2([0,T]; \bf H^m(\Oe))$ and  $\phi \in C^2([0,T];H^n(\Oa))$, with $m \ge 2$, $n\ge 2$. Let $(\bf u_h,\phi_h) \in C^2([0,T];\bf V_{\! h}^e) \times C^2([0,T]; V_h^a)$ be the corresponding solution of the semi-discrete problem \eqref{coupled_dG}, with sufficiently large penalty parameters $\alpha$ and $\beta$ in \mbox{\eqref{stabiliz.a}--\eqref{stabiliz.b}}. Then, the following bound holds for the discretization error
$\bf E(t)  = (\bf e_e(t), e_a(t))  = (\bf u(t) - \bf u_h(t), \phi(t)  - \phi_h(t))$:
%\hspace{-1cm}
\begin{equation}
\begin{alignedat}{2}\label{eq:error-estimate}
%\begin{aligned}
%\hspace{-.5cm}
&
\sup_{t\in[0,T]} \|\bf E(t)\|_\E^2 \lesssim \sup_{t\in[0,T]}\left(\sum_{\k\in\cc T_h^e}\frac{h_\k^{2\min(m,p_{e,\k}+1)-2}}{p_{e,\k}^{2m -3}}\left(
\|\mathscr E \dot{\bf u}\|_{m,\cc K}^2 + \|\mathscr E \bf u\|_{m,\cc K}^2\right)
\right.  &&\\
&\qquad\qquad\qquad\qquad\left.
+\sum_{\k\in\cc T_h^a}\frac{h_\k^{2\min(n,p_{a,\k}+1)-2}}{p_{a,\k}^{2n-3}}
\left( \|\mathscr E \dot\phi\|_{n,\cc K}^2 +  \|\mathscr E \phi\|_{n,\cc K}^2\right)
\right) + &&\\
%\hspace{2cm}
&\qquad\qquad\quad \ \
+ \int_0^T\!\left( \sum_{\k \in\cc T_h^e}\frac{h_\k^{2\min(m,p_{e,\k}+1)-2}}{p_{e,\k}^{2m-3}}\left(
\|\mathscr E \ddot{\bf u}\|_{m,\cc K}^2 +
 \|\mathscr E \dot{\bf u}\|_{m,\cc K}^2 + \|\mathscr E \bf u\|_{m,\cc K}^2\right) \right. 
  &&\\
%\hspace{1.5cm}
&\qquad\qquad\qquad\left. 
+ \sum_{\k\in\cc T_h^a}\frac{h_\k^{2\min(n,p_{a,\k}+1)-2}}{p_{a,\k}^{2n-3}}\left(
\|\mathscr E \ddot{\phi}\|_{n,\cc K}^2 +
\|\mathscr E \dot{\phi}\|_{n,\cc K}^2 + 
\|\mathscr E \phi\|_{n,\cc K}^2\right)
\right)\d \tau. &&
\end{alignedat}
\end{equation}
\end{theorem}
\begin{corollary}[\emph{A priori} error estimate in the energy norm]
Under the hypotheses of Theorem~\ref{error-estimate}, assume that $h \simeq h_\k$ for any $\k \in \cc T_h$, $p_{e,\k} = p_e$ for any $\k \in \cc T_h^e$, and $p_{a,\k} = p_a$ for any $\k \in \cc T_h^a$. Then, if $({\bf u},\phi) \in C^2([0,T]; \bf H^m(\Oe)) \times C^2([0,T];H^n(\Oa))$ with $m\ge p_e+1$ and $n \ge p_a+1$, the error estimate \eqref{eq:error-estimate} reads
%\begin{multline*}
%\hspace{-.4cm}
%\sup_{t\in[0,T]} \|\bf E_h(t)\|_\E^2 \lesssim \frac{h^{2p_{e,\k}}}{p_{e,\k}^{2m-3}}
%\left( \sup_{t\in[0,T]} \left( \|\dot{\bf u}(t)\|_{m,\Oe}^2 + \|\bf u(t)\|_{m,\Oe}^2\right)\right. \\
%\left. +\int_0^T \left(
%\|\ddot{\bf u}(t)\|_{m,\Oe}^2 + \|\dot{\bf u}(t)\|_{m,\Oe}^2 + \|\bf u(t)\|_{m,\Oe}^2
%\right)\d t
%\right)
%\\ + \frac{H^{2p_{a,\k}}}{p_{a,\k}^{2n-3}}
%\left( \sup_{t\in[0,T]} \|\phi(t)\|_{n,\Oa}^2 + \int_0^T\big(
%\|\dot{\phi}(t)\|_{n,\Oa}^2 + 
%\|\phi(t)\|_{n,\Oa}^2\big)\,\d t
%\right).
%\end{multline*}
\begin{equation}\label{errore'}
\begin{aligned}
\hspace{-.1cm}
\sup_{t\in[0,T]} \|\bf E(t)\|_\E^2 \lesssim & \ \frac{h^{2p_e}}{p_e^{2m-3}}
\Bigg( \sup_{t\in[0,T]} \left( \|\dot{\bf u}\|_{m,\Oe}^2 + \|\bf u\|_{m,\Oe}^2\right)
%& 
%\qquad\qquad
 +\int_0^T\!\! \left( \|\ddot{\bf u}\|_{m,\Oe}^2 + \|\dot{\bf u}\|_{m,\Oe}^2 + \|\bf u\|_{m,\Oe}^2
\right)\d t
\Bigg)
\\ & + \frac{h^{2p_a}}{p_a^{2n-3}}
\Bigg( \sup_{t\in[0,T]} \left( \|\dot\phi\|_{n,\Oa}^2 + \|\phi\|_{n,\Oa}^2 \right)
%\\ & \qquad\qquad 
+ \int_0^T \! \! \big(
\|\ddot{\phi}\|_{n,\Oa}^2 +
\|\dot{\phi}\|_{n,\Oa}^2 + 
\|\phi\|_{n,\Oa}^2\big)\,\d t
\Bigg). %\\[1ex]
\end{aligned}
\end{equation}
\end{corollary}
\begin{proof}[Proof of Theorem \ref{error-estimate}]
It is easy to see that the semi-discrete formulation \eqref{coupled_dG} is \emph{strongly consistent}, i.e., the exact solution $(\bf u, \phi)$ satisfies \eqref{coupled_dG} for any $t\in(0,T]$:
\begin{equation}\label{consistency}
\begin{multlined}
%\hspace{-.4cm}
 (\rho_e\ddot{\bf u},\bf v)_{\Omega_e} + (c^{-2}\rho_a\ddot\phi,\psi)_{\Omega_a}  %\\ 
%& \qquad
+ (2\rho_e\z \dot{\bf u},\bf v)_{\Omega_e} +
 (\rho_e\z^2\bf u,\bf v)_{\Omega_e} \\[4pt] %\\
+ \cc A_h^e(\bf u,\bf v)  + \cc A_h^a(\phi,\psi)% \\
%& \qquad\qquad\qquad
+  \cc I_h^e(\dot\phi, \bf v) + \cc I_h^a(\dot{\bf u}, \psi)  \\
\hspace{4cm} = ({\bf f_e}, \bf v)_{\Omega_e} + (\rho_a f_a,\psi)_\Oa,
\quad \forall (\bf v,\psi)\in \bf V_{\! h}^e\times V_h^a.
%\end{aligned}
\end{multlined}
\end{equation}
Subtracting \eqref{coupled_dG} from the above identity, we obtain the \emph{error equation}:
\begin{equation}\label{error_equation}
\begin{multlined}
%\begin{aligned}
%\hspace{-.8cm}
(\rho_e\ddot{\bf e}_e,\bf v)_{\Omega_e} + (c^{-2}\rho_a\ddot e_a,\psi)_{\Omega_a} 
%& \qquad
+ (2\rho_e\z \dot{\bf e}_e,\bf v)_{\Omega_e} +
 (\rho_e\z^2\bf e_e,\bf v)_{\Omega_e} \\[4pt]
 %\\
%
%&\qquad\qquad 
%\hspace{.8cm}
+ \cc A_h^e(\bf e_e,\bf v)  + \cc A_h^a(e_a,\psi) %\\
%& \qquad\qquad\qquad
+  \cc I_h^e(\dot e_a, \bf v) + \cc I_h^a(\dot{\bf e}_e, \psi)  = 0,
\ \ \forall (\bf v,\psi)\in \bf V_{\! h}^e\times V_h^a.
%\end{aligned}
\end{multlined}
\end{equation}
We next decompose the error $\bf E = (\bf e_e, e_a)$ as follows: $\bf E = \bf E_I - \bf E_h$, with $\bf E_I = (\bf e_I, e_I) = (\bf u - \bf u_I, \phi - \phi_I)$, and $\bf E_h = (\bf e_h, e_h) =  (\bf u_h - \bf u_I, \phi_h - \phi_I)$, $(\bf u_I,\phi_I) \in \bf V_{\! h}^e\times V_h^a$ being the interpolants defined as in Lemma \ref{interpolation_estimates}. By taking as test functions $(\bf v, \psi) = (\dot{\bf e}_h, \dot e_h)$, the above identity reads then
\begin{multline*}
(\rho_e\ddot{\bf e}_h,\dot{\bf e}_h)_{\Omega_e} + (c^{-2}\rho_a\ddot e_h,\dot e_h)_{\Omega_a} 
%& \qquad
+ (2\rho_e\z \dot{\bf e}_h,\dot{\bf e}_h)_{\Omega_e} +
 (\rho_e\z^2\bf e_h,\dot{\bf e}_h)_{\Omega_e} \\
 %\\
%
%&\qquad\qquad 
+ \cc A_h^e(\bf e_h,\dot{\bf e}_h)  + \cc A_h^a(e_h,\dot e_h) %\\
%& \qquad\qquad\qquad
+  \cc I_h^e(\dot e_h, \dot{\bf e}_h) + \cc I_h^a(\dot{\bf e}_h, \dot e_h)  \\
= (\rho_e\ddot{\bf e}_I,\dot{\bf e}_h)_{\Omega_e} + (c^{-2}\rho_a\ddot e_I,\dot e_h)_{\Omega_a} 
%& \qquad
+ (2\rho_e\z \dot{\bf e}_I,\dot{\bf e}_h)_{\Omega_e} +
 (\rho_e\z^2\bf e_I,\dot{\bf e}_h)_{\Omega_e} \\
+ \cc A_h^e(\bf e_I,\dot{\bf e}_h)  + \cc A_h^a(e_I,\dot e_h) %\\
+  \cc I_h^e(\dot e_I, \dot{\bf e}_h) + \cc I_h^a(\dot{\bf e}_I, \dot e_h).
\end{multline*}
Using the Cauchy--Schwarz inequality to bound the
terms on the right-hand side, the above estimate can be rewritten as
\begin{multline*}
\frac{1}{2}\frac{\d}{\d t} \left( \| \bf E_h \|_\E^2 - 2\la \av{\bm\sigma_h(\bf e_h)}, \jm{\bf e_h} \ra_{{\cc F}_h^e} - 2 \la \av{\rho_a \gradh e_h},\jm{e_h} \ra_{{\cc F}_h^a} \right) \\ + 2 \|\rho_e^{\nf12}\z^{\nf12} \dot{\bf e}_h\|_\Oe^2 %\\
\le \|\dot{\bf e}_I\|_{\E_e}\|\bf e_h\|_{\E_e} + \|\dot e_I\|_{\E_a}\|e_h\|_{\E_a} + 2\|\rho_e^{\nf12}\z^{\nf12}\dot{\bf e}_I\|_\Oe
\|\rho_e^{\nf12}\z^{\nf12}\dot{\bf e}_h\|_\Oe \\
+ \cc A_h^e(\bf e_I,\dot{\bf e}_h)  + \cc A_h^a(e_I,\dot e_h) %\\
+  \cc I_h^e(\dot e_I, \dot{\bf e}_h) + \cc I_h^a(\dot{\bf e}_I, \dot e_h) + (\rho\z^2\bf e_I,\dot{\bf e}_h)_\Oe.
\end{multline*}
This inequality can be further manipulated by observing that $$2\|\rho_e^{\nf12}\z^{\nf12}\dot{\bf e}_I\|_\Oe \|\rho_e^{\nf12}\z^{\nf12}\dot{\bf e}_h\|_\Oe \le \|\rho_e^{\nf12}\z^{\nf12}\dot{\bf e}_I\|_\Oe^2 + \|\rho_e^{\nf12}\z^{\nf12}\dot{\bf e}_h\|_\Oe^2;$$ thereby we obtain
\begin{multline*}
\frac{1}{2}\frac{\d}{\d t} \left( \| \bf E_h \|_\E^2 - 2\la \av{\bm\sigma_h(\bf e_h)}, \jm{\bf e_h} \ra_{{\cc F}_h^e} - 2 \la \av{\rho_a \gradh e_h},\jm{e_h} \ra_{{\cc F}_h^a} \right) \\ + \|\rho_e^{\nf12}\z^{\nf12} \dot{\bf e}_h\|_\Oe^2 %\\
%\hspace{-1cm}
\le \|\dot{\bf e}_I\|_{\E_e}\|\bf e_h\|_{\E_e} + \|\dot e_I\|_{\E_a}\|e_h\|_{\E_a} + \|\rho_e^{\nf12}\z^{\nf12}\dot{\bf e}_I\|_\Oe^2
\\
+ \cc A_h^e(\bf e_I,\dot{\bf e}_h)  + \cc A_h^a(e_I,\dot e_h) %\\
%& \qquad\qquad\qquad
+  \cc I_h^e(\dot e_I, \dot{\bf e}_h) + \cc I_h^a(\dot{\bf e}_I, \dot e_h) + (\rho\z^2\bf e_I,\dot{\bf e}_h)_\Oe.
\end{multline*}
Since $\|\rho_e^{\nf12}\z^{\nf12} \dot{\bf e}_h\|_\Oe^2 \ge 0$, integrating in time between $0$ and $t$, using Lemma \ref{bounds}, and choosing the projections of the initial data such that $\bf e_h(0) = \bf u_{0,h} - (\bf u_0)_I = \bf 0$ and $e_h(0) = \phi_{0,h} - (\phi_0)_I = 0$, we get
\begin{multline}\label{stima_EH}
\| \bf E_h \|_\E^2 \lesssim \int_0^t \left( \|\dot{\bf e}_I\|_{\E_e}\|\bf e_h\|_{\E_e} + \|\dot e_I\|_{\E_a}\|e_h\|_{\E_a} \right) \d\tau+
\int_0^t \|\rho_e^{\nf12}\z^{\nf12}\dot{\bf e}_I\|_\Oe^2 \d\tau % \\
+ \int_0^t (\rho_e\z^2\bf e_I,\dot{\bf e}_h)_\Oe \d \tau \\ 
%\qquad
+ \int_0^t \big(\cc A_h^e(\bf e_I,\dot{\bf e}_h)  + \cc A_h^a(e_I,\dot e_h)\big) \d\tau %\\
+ \int_0^t \big( \cc I_h^e(\dot e_I, \dot{\bf e}_h) + \cc I_h^a(\dot{\bf e}_I, \dot e_h) \big) \d \tau.
\end{multline}
Performing integration by parts in time between $0$ and $t$ on the third term on the right-hand side, and using the fact that $\bf e_h(0) = \bf 0$, $e_h(0) =  0$ and the definition \eqref{energy_norm} of the energy norm yields
$$
\begin{aligned}
\int_0^t (\rho_e\z^2\bf e_I, \dot{\bf e}_h)_\Oe\d\tau 
& = (\rho_e\z^2\bf e_I(t),\bf e_h(t))_\Oe - \int_0^t (\rho_e\z^2\dot{\bf e}_I,\bf e_h)_\Oe\d\tau \\
& 
\lesssim \|\bf e_I\|_{\E_e} \|\bf e_h\|_{\E_e} + \int_0^t \|\dot{\bf e}_I\|_{\E_e} \|\bf e_h\|_{\E_e} \d\tau.
\end{aligned}
$$
Analogously, using the continuity of bilinear forms $\cc A_h^e$ and $\cc A_h^a$ expressed by \eqref{cont_estesa}, and the definition \eqref{energy_norm} of the energy norm, we obtain
$$
\begin{aligned}
\int_0^t \big( \cc A_h^e(\bf e_I,\dot{\bf e}_h)  + \cc A_h^a(e_I,\dot e_h)\big) \d\tau   = & \
 \cc A_h^e(\bf e_I(t),{\bf e}_h(t))  + \cc A_h^a(e_I(t), e_h(t)) \\ 
&  
\ 
- \int_0^t \big(\cc A_h^e(\dot{\bf e}_I,{\bf e}_h)  + \cc A_h^a({\dot e}_I, e_h)\big) \d\tau \\
 & \lesssim \ \vvvert{\bf e_I}\vvvert_\dGe \|\bf e_h\|_{\E_e} + \vvvert{e_I}\vvvert_\dGa \|e_h\|_{\E_a} \\
 & \  +  \int_0^t \big( \tpn{\dot{\bf e}_I}_\dGe \|\bf e_h\|_{\E_e} + \tpn{\dot e_I}_\dGa \|e_h\|_{\E_a} \big)\d\tau.
\end{aligned}
$$

We now seek a bound on the fifth term on the right-hand side of \eqref{stima_EH}. Focusing on the bilinear form $\cc I_h^e$ (cf.~definition \eqref{discr_bilin_forms}), we have
\begin{equation*}
\begin{aligned}
\cc I_h^e(\dot e_I,\dot{\bf e}_h) & = \sum_{F\in \cc F_{h,\mr I}} (\rho_a \dot e_I \bf n_e, \dot{\bf e}_h)_F \le \sum_{F\in \cc F_{h,\mr I}} \| \rho_a \dot e_I\|_F \|\dot{\bf e}_h\|_F %\\ & 
\lesssim
\sum_{\k^e \in \cc T_{h,\mr I}^e, \, \k^a \in \cc T_{h,\mr I}^a} \|\dot e_I\|_{\de\k^a} \|\dot{\bf e}_h\|_{\de\k^e} \\
& 
\lesssim \sum_{\k^e \in \cc T_{h,\mr I}^e, \, \k^a \in \cc T_{h,\mr I}^a} p_{e,\k^e}h_{\k^e}^{-\nf12} \|\dot e_I\|_{\de\k^a} \|\dot{\bf e}_h\|_{\k^e} %\\ &\lesssim \bigg(\sum_{\k \in \cc T_{h,\mr I}^a} p_{a,\k}h_{\k}^{-\nf12} \|\dot e_I\|_{\de\k}\bigg)
 \|\bf e_h\|_{\E_e},
\end{aligned}
\end{equation*}
where we have used the Cauchy--Schwarz inequality, the trace-inverse inequality \eqref{eq:trace.inv.ineq}, the definition \eqref{energy_norm} of the energy norm, and, in the last bound, Assumption \ref{local_bv} on $hp$-local bounded variation.
Hence, we have
\begin{equation}\label{Ja}
\int_0^t \cc I_h^e(\dot e_I, \dot{\bf e}_h) \,\d\tau \lesssim \int_0^t \bigg(\sum_{\k \in \cc T_{h,\mr I}^a} p_{a,\k}h_{\k}^{-\nf12} \|\dot e_I\|_{\de\k}\bigg) \|\bf e_h\|_{\E_e}\d\tau
\equiv \int_0^t \cc J_h^a(\dot e_I)\|\bf e_h\|_{\E_e}\d\tau.
\end{equation}
Recalling that $\cc I_h^a(\dot{\bf e}_I,\dot e_h) = - \cc I_h^e(\dot e_h,\dot{\bf e}_I)$, with completely analogous arguments we obtain
\begin{equation}\label{Je} \int_0^t \cc I_h^a(\dot{\bf e}_I,\dot e_h) \, \d\tau \lesssim \int_0^t \bigg(\sum_{\k \in \cc T_{h,\mr I}^e} p_{e,\k}h_{\k}^{-\nf12} \|\dot{\bf e}_I\|_{\de\k}\bigg) 
\|e_h\|_{\E_a}\d\tau
\equiv \int_0^t \cc J_h^e(\dot{\bf e}_I)\|e_h\|_{\E_a}\d\tau.
\end{equation}
Substituting the above bounds into \eqref{stima_EH}, we get
\begin{align*}
\|\bf E_h\|_\E^2 \lesssim & \left(\|\bf e_I\|_{\E_e} +  \tpn{\bf e_I}_\dGe \right) \|\bf e_h\|_{\E_e} +\tpn{e_I}_\dGa \|e_h\|_{\E_a} + \int_0^t \|\rho_e^{\nf12}\z^{\nf12}\dot{\bf e}_I\|_\Oe^2\d\tau  \\ 
& + \int_0^t  \big(\|\dot{\bf e}_I\|_{\E_e} + \tpn{\dot{\bf e}_I}_\dGe + \cc J_h^e(\dot e_I)\big) \|\bf e_h\|_{\E_e} \d\tau%\\ 
%& 
+ \int_0^t \big(\|\dot e_I\|_{\E_a} + \tpn{\dot e_I}_{\dGa} + \cc J_h^a(\dot{\bf e}_I)\big)
\| e_h\|_{\E_a} \d\tau.
\end{align*}
Observe now that $\|\bf e_h\|_{\E_e} \le \|\bf E_h\|_\E$ and $\|e_h\|_{\E_a}\le \|\bf E_h\|_\E$. Thanks to Young's inequality we have
$$\begin{aligned}
\left(\|\bf e_I\|_{\E_e} +  \tpn{\bf e_I}_\dGe \right) \|\bf e_h\|_{\E_e} & \le \frac{\epsi}{2}\|\bf e_h\|_{\E_e}^2 + \frac{1}{2\epsi}\left(\|\bf e_I\|_{\E_e} +  \tpn{\bf e_I}_\dGe \right)^2 \\ & \le \frac{\epsi}{2}\|\bf E_h\|_{\E}^2 + \frac{1}{\epsi}\left(\|\bf e_I\|_{\E_e}^2 +  \tpn{\bf e_I}_\dGe^2 \right), \\[5pt]
\tpn{e_I}_\dGa \|e_h\|_{\E_a} & \le \frac{\delta}{2} \|\bf E_h\|_\E^2 + \frac{1}{2\delta} \tpn{e_I}_\dGa^2.
\end{aligned}$$
Choosing $\epsi$ such that $1-\frac{1}{2}C\epsi > 0$ and $\delta$ such that $1-\frac{1}{2}C(\delta+\epsi) > 0$, $C$ being the hidden constant in \eqref{stima_EH},
we infer that
$$
\begin{aligned}
\|\bf E_h\|_\E^2  \lesssim & \ \|\bf e_I\|_{\E_e}^2 + \tpn{\bf e_I}_\dGe^2 + \tpn{e_I}_\dGa^2 + \int_0^t \|\rho_e^{\nf12}\z^{\nf12}\dot{\bf e}_I\|_\Oe^2\d\tau
\\ & 
+ \int_0^t \left( \|\dot{\bf e}_I\|_{\E_e} + \tpn{\dot{\bf e}_I}_\dGe + \cc J_h^a(\dot e_I) %\right.\\
%&\qquad 
%\left.
+\|\dot e_I\|_{\E_a} + \tpn{\dot e_I}_{\dGa} + \cc J_h^e(\dot{\bf e}_I)
     \right) \|\bf E_h\|_\E \ \d\tau.
\end{aligned}
$$
Upon setting
$$G= \sup_{t\in [0,T]} \left( \|\bf e_I\|_{\E_e}^2 + \tpn{\bf e_I}_\dGe^2 + \tpn{e_I}_\dGa^2 \right) + \int_0^T \|\rho_e^{\nf12}\z^{\nf12}\dot{\bf e}_I\|_\Oe^2\d\tau,$$
and applying Gronwall's Lemma \cite[p.~28]{quarteroni} along with Jensen's inequality, we get
\begin{equation}\label{final_ineq}
%\begin{aligned}
\|\bf E_h\|_\E^2 \lesssim G + \int_0^T \left(\|\dot{\bf e}_I\|_{\E_e}^2 + \tpn{\dot{\bf e}_I}_\dGe^2  + \cc J_h^e(\dot{\bf e}_I)^2 + \|\dot e_I\|_{\E_a}^2 + \tpn{\dot e_I}_{\dGa}^2 + \cc J_h^a(\dot e_I)^2 \right)\d\tau.
%\end{aligned}
\end{equation}
Owing to $hp$-approximation boundary estimates \cite[Lemma 33]{cangiani-book}, we infer that
$$
\begin{aligned}
\cc J_h^a (\dot e_I) & \lesssim
 \sum_{\k \in \cc T_{h,\mr I}^a} \frac{h_\k^{\min(p_{a,\k}+1,n)-1}}{p_{a,\k}^{n-\nf32}} \| \mathscr E \dot\phi\|_{n,\cc K},
 \\
 \cc J_h^e(\dot{\bf e}_I) & \lesssim
 \sum_{\k \in \cc T_{h,\mr I}^e} \frac{h_\k^{\min(p_{e,\k}+1,m)-1}}{p_{e,\k}^{m-\nf32}} \| \mathscr E \dot{\bf u}\|_{m,\cc K}
\end{aligned}
$$
(cf.~\eqref{Ja} and \eqref{Je}).
Applying the bounds of Lemma \ref{interpolation_estimates} to estimate the energy- and $\dG$-norms in the right-hand side of \eqref{final_ineq}, observing that $\|\bf E(t)\|_{\E}^2 \le 2(\|\bf E_h(t)\|_\E^2 + \|\bf E_I(t)\|_\E^2)$ $\forall t\in [0,T]$, applying again the bounds of Lemma \ref{interpolation_estimates} to estimate the second addend, and taking the supremum over $[0,T]$ of the resulting estimate, the thesis follows.
\end{proof}

\section{Fully discrete formulation}\label{sec:fully.discrete}
By choosing bases for the discrete spaces $\bf V_{\! h}^e$ and $V_h^a$, the semi-discrete algebraic formulation of problem \eqref{coupled_dG} reads
\begin{equation}\label{time_discr}
\begin{dcases}
\begin{aligned}
\sF M_e^1 \ddot{\mathsf U}(t)  + \sF M_e^2 \dot{\mathsf U}(t) + (\sF M_e^3 + \mathsf A_e)% + \sF R_e) 
\mathsf U(t) + \mathsf C_e \dot{\mathsf \Phi} (t)
& = \mathsf F_e(t), \quad t \in (0,T], \\
\mathsf M_a\ddot{\mathsf \Phi}(t) +
 \mathsf A_a\mathsf \Phi(t) + \mathsf C_a \dot{\mathsf U}(t) & = \mathsf F_a(t), \quad t \in (0,T],\\
\mathsf U(0) & = \mathsf U^0, \\
\dot{\mathsf U}(0) & = \mathsf V^0, \\
\mathsf{\Phi}(0) & = \mathsf \Phi^0,\\
\dot{\mathsf\Phi}(0) & = \mathsf \Psi^0,\\
\end{aligned}
\end{dcases}
\end{equation}
where vectors $\sF U(t)$ and $\sF \Phi(t)$ represent the expansion coefficients of $\bm u_h(t)$ and $\phi_h(t)$ in the chosen bases.
Analogously, $\sF M_e^1$, $\sF M_e^2$, $\sF M_e^3$, $\sF A_e$, and $\sF C_e$ are the matrices stemming from the bilinear forms
$$(\rho_e\bf u,\bf v)_{\Oe}, \ (2\rho_e\z \bf u, \bf v)_{\Oe}, \ (\rho_e\z^2 \bf u, \bf v)_{\Oe}, \
\cc A_h(\bf u, \bf v), \ \cc I_h^e(\psi,\bf v),$$
respectively, and $\sF M_a$, $\sF A_a$, $\sF C_a \equiv - \sF C_e^{\sF T}$ represent the bilinear forms
$$(c^{-2}\rho_a\phi,\psi)_{\Oa},\ \cc A_h^a(\phi,\psi), \ \cc I_h^a(\bf v,\psi),$$ respectively. Finally, $\sF F_e(t)$ and $\sF F_a(t)$ are the vector representations of linear functionals $(\bf f_e, \bf v)_\Oe$ and $(\rho_a f_a,\psi)_\Oa$, respectively.

To fully discretize \eqref{coupled_dG}, we employ a time marching method based on centered finite-difference, widely employed for the numerical simulation of wave propagation, namely, the \emph{leap-frog} scheme. % Denoting by $N_{\rm dof}^e$ the total number of elastic degrees of freedom, and by $N_{\rm dof}^a$ the total number of acoustic degrees of freedom, vectors $\mathsf U(t) \in \bb R^{N_{\rm dof}^e}$ and $\mathsf \Phi(t) \in \bb R^{N_{\rm dof}^a}$ contain, for any time $t$, the expansion coefficients of the semi-discrete solution in the chosen set of basis functions. Vectors $\sF U^0$, $\sF V^0$, $\sF \Phi^0$, and $\sF \Psi^0$ constitute the set of initial conditions.
We now subdivide the time interval $(0,T]$ into $N_T$ subintervals of amplitude $\Delta t = T/N_T$ and we denote by $\sF U^i \approx \sF U(t_i)$ and $\sF \Phi^i \approx \sF \Phi(t_i)$ the approximations of $\sF U$ and $\sF \Phi$ at time $t_i = i\Delta t$, $i \in \{1, \dots, N_T\}$. The centered finite-difference method reads then
\begin{equation*}%\label{leap-frog}
\begin{aligned}
\begin{bmatrix}
\sF M_e^1+\frac{\Delta t}{2}\sF M_e^2 & \frac{\Delta t}{2} \sF C_e \\
- \frac{\Delta t}{2} \sF C_e^{\sF T} & \sF M_a
\end{bmatrix} 
\begin{bmatrix} \sF U^{n+1} \\ \sF \Phi^{n+1} \end{bmatrix}
 = & 
\begin{bmatrix}
-\sF M_e^1+\frac{\Delta t}{2}\sF M_e^2 & \frac{\Delta t}{2} \sF C_e \\
- \frac{\Delta t}{2} \sF C_e^{\sF T} & - \sF M_a
\end{bmatrix} 
\begin{bmatrix} \sF U^{n-1} \\ \sF \Phi^{n-1} \end{bmatrix}
 \\
& +
\begin{bmatrix}
2\sF M_e^1 - {\Delta t^2}(\sF A_e +\sF M_e^3) & \sF 0 \\ 
\sF 0 & 2\sF M_a - \Delta t^2 \sF A_a
\end{bmatrix}  
\begin{bmatrix} \sF U^n \\ \sF \Phi^n \end{bmatrix}
+ \Delta t^2 \begin{bmatrix}   \sF F_e^n  \\  \sF F_a^n  \end{bmatrix},
\end{aligned}
\end{equation*}
for $n \in \{1,\dots,N_T-1\}$, where $\sF U^1 = \sF U^0 + \Delta t \, \sF V^0$, $\sF \Phi^1 = \sF \Phi^0 + \Delta t \, \sF \Psi^0$, and $\sF F_{e(a)}^n = \sF F_{e(a)}(t_n)$.
Let us remark that the centered finite-difference method is an explicit second-order-accurate scheme; thus, to ensure its numerical stability, a Courant-Friedrich-Lewy (CFL) condition has to be satisfied (see \cite{quarteroni-valli}).
%%%%
\section{Numerical examples}\label{sec:numerical.examples}
In this section we solve problem \eqref{strong_form} for $\z = 0$ in the rectangle $\Omega = (-1,1) \times (0,1)$ on polygonal meshes such as the one represented in Figure \ref{mesh}. Numerical experiments have been carried out both to test $hp$-convergence (besides validating numerically estimate \eqref{errore'} by computing the \mbox{dG-norm} of the error, we also check convergence of the method in the $L^2$-norm) and to simulate a problem of physical interest, where the system is excited by a point source load in the acoustic domain. In all cases, we assume that $\Oe = (-1,0)\times (0,1)$ is occupied by an \emph{isotropic} material, i.e., $\bb C$ is such that 
$\bm\sigma(\bf u) = 2\mu\,\bm\eps(\bf u) + \lambda(\div\bf u)\bf I$,
with $\mu$ and $\lambda$ the Lamé coefficients, both constant over $\Oe$, and $\Oa = (0,1) \times (0,1)$ is occupied by a fluid with constant density $\rho_a$. The interface is thus given by $\Gamma_{\mr I} = \{0\}\times (0,1)$. Meshes have been generated using \texttt{PolyMesher} \cite{polymesher}. The timestep will be precised depending on the case under consideration. In all of the numerical experiments, all the physical quantities involved are supposed to be dimensionless. In Sections~\ref{sec:hom_bc} and \ref{sec:nonhom_bc} we choose, as in \cite{monkola-sanna}, $\mu = 26.29$, $\lambda = 51.20$, $\rho_e = 2.7$, $\rho_a = 1$, and $c=1$.
\subsection{Test case 1}\label{sec:hom_bc}
In this test case, the right-hand sides $\bf f_e$ and $f_a$ are chosen so that the exact solution is given by
\begin{equation}\label{hom_bc}
\begin{aligned}
\bf u(x,y;t) = x^2 \cos(\sqrt{2}\pi t)\cos\left(\frac{\pi}{2} x\right)\sin(\pi y) \, \wh{\bf u},\ \
 \phi(x,y;t) = x^2 \sin(\sqrt{2}\pi t)\sin(\pi x)\sin(\pi y),
\end{aligned}
\end{equation}
where $\wh{\bf u} = (1,1)$. The timestep is here set to $\Delta t = 10^{-4}$, so that the error due to time integration is negligeable, and the final time is set to $T=1$.
Notice that, in this case, both the left- and right-hand sides of the transmission conditions on $\Gamma_\mr I$ (cf.~\eqref{strong_form}) vanish,
as well as the unknowns $\bf u$ and $\phi$ themselves.
\begin{figure}
\centering
\includegraphics[keepaspectratio=true,scale=.45]{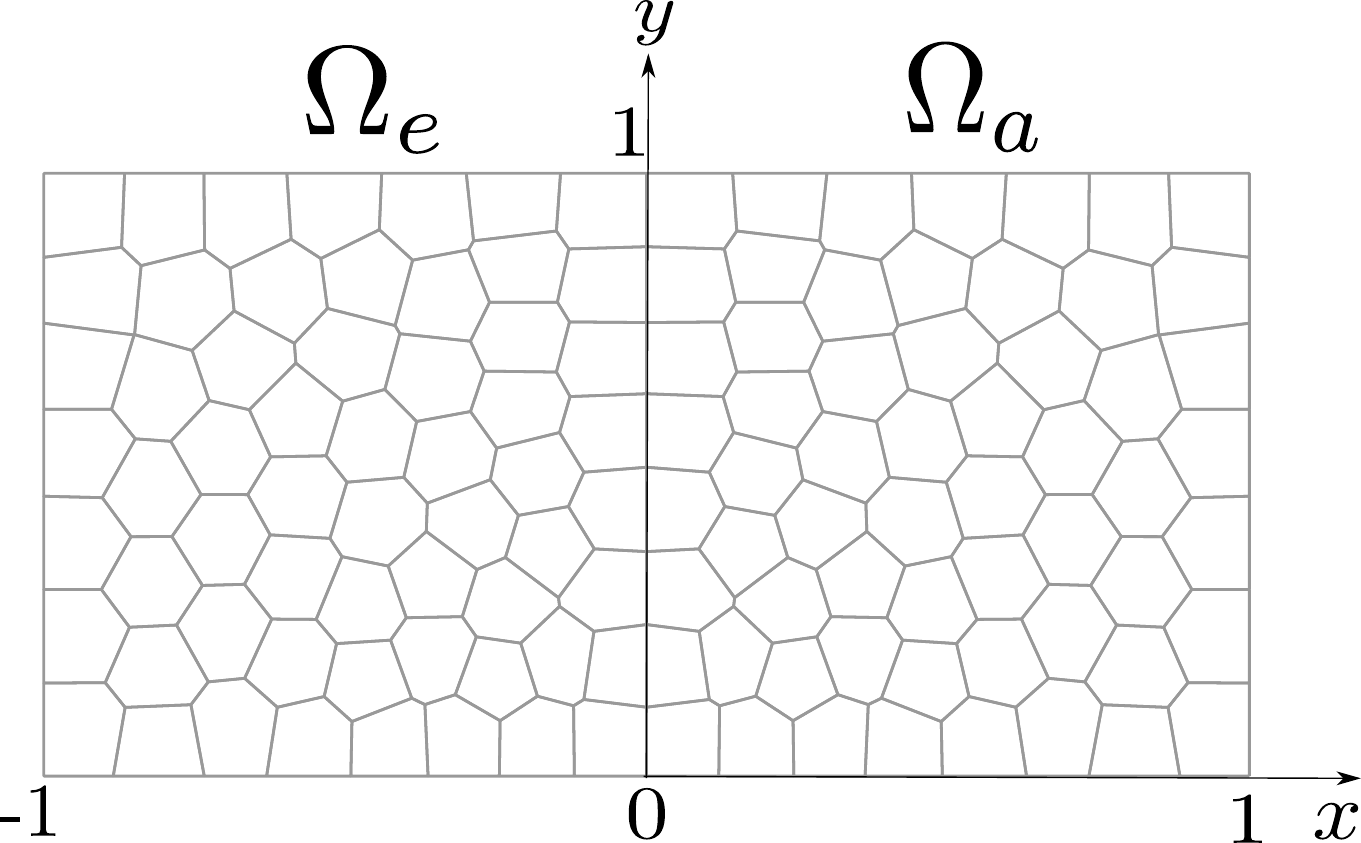}
\caption{Computational domain and mesh made up by $120$ polygons.}
\label{mesh}
\end{figure}
%
%\subsubsection*{\boldmath $h$-refinement}
Figure \ref{errore_h} shows convergence results in the dG- and $L^2$-norms respectively, for four nested, sequentially refined polygonal meshes, %(the coarsest mesh containing 80 elements, the finest 5120), 
when polynomials of uniform degree $p=2$ are employed. The numerical results concerning the dG-error show asymptotic convergence rates that match those predicted by estimate \eqref{errore'}. Also, as it is typical for dG methods, the $L^2$-error turns out to converge in $h^{p+1}$ (see, e.g., \cite[Theorem 2]{antonietti-ferroni-mazzieri-quarteroni} for the case of the elastodynamics equation).
%\subsubsection*{\boldmath $p$-refinement}

Figure \ref{errore_p} shows convergence results in a semilogarithmic scale, in the dG- and $L^2$-norms respectively, for a fixed mesh given by 300 elements and a uniform polynomial degree ranging from 1 to 5. Since the exact solution is analytical, as expected, the error undergoes an exponential decay.%, as predicted., for instance, by the theory of Spectral Element methods (see e.g. \cite[Chapter 10]{quarteroni}).
\begin{figure}
\centering
\subfloat[$\|\bf u - \bf u_h\|_{\dGe}$ and $\|\phi - \phi_h\|_\dGa$ vs.~$h$ at $T=1$]{\includegraphics[scale=.45]{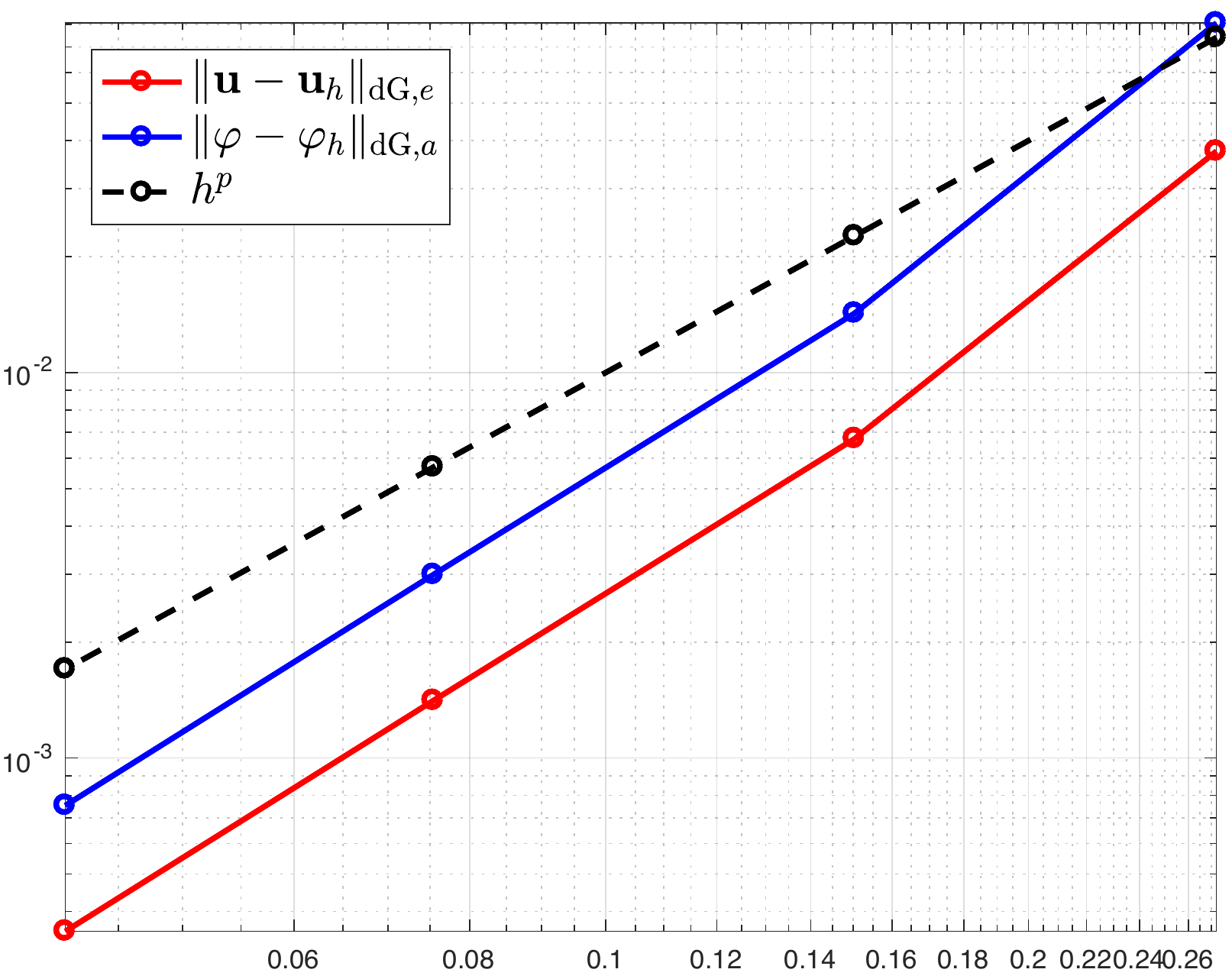}}
\hspace{.01cm}
\subfloat[$\|\bf u - \bf u_h\|_\Oe$ and $\|\phi - \phi_h\|_\Oa$ vs.~$h$ at $T=1$]{\includegraphics[scale=.45]{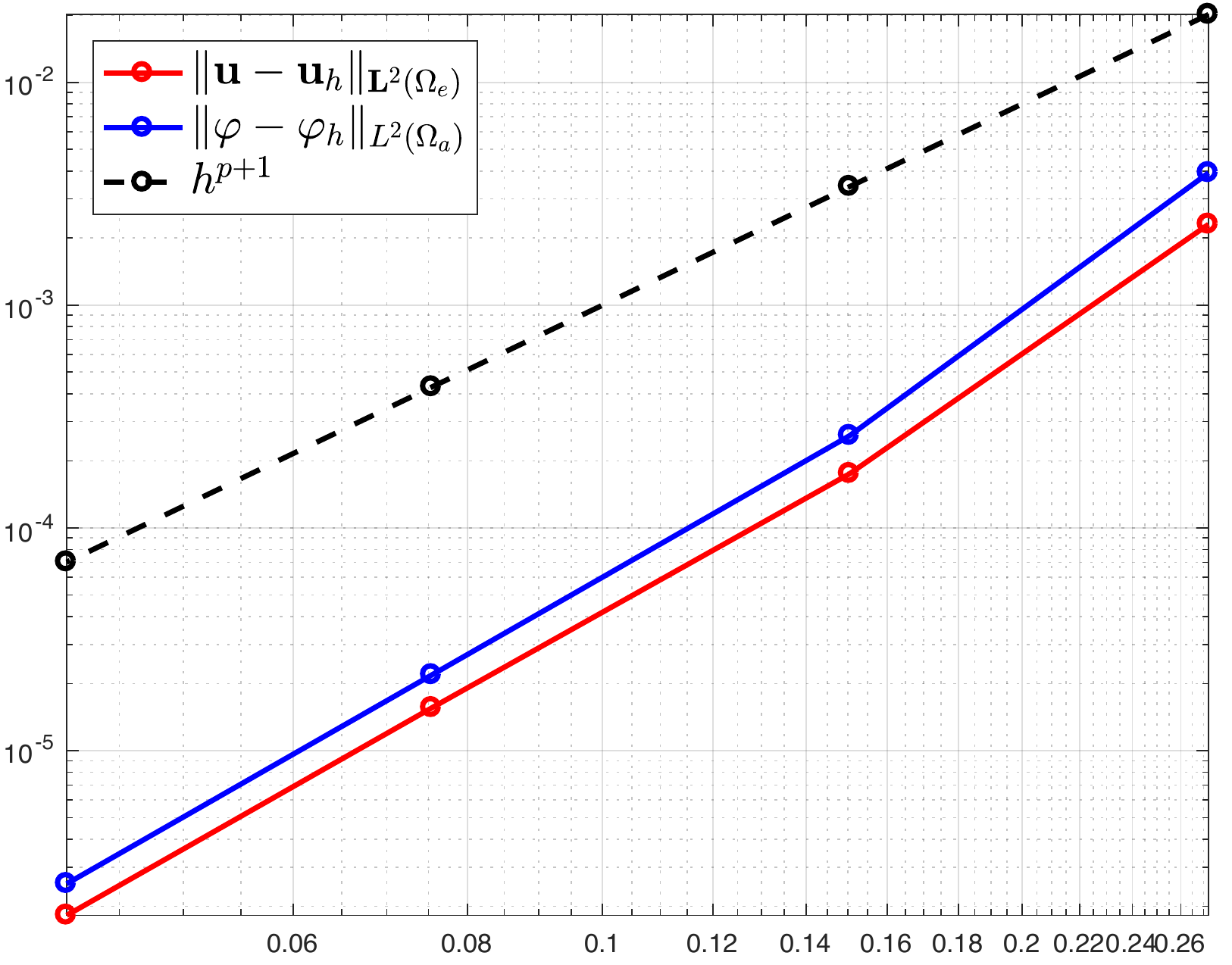}}
\caption{Test case 1. dG-error and $L^2$-error vs.~$h$ for four sequentially refined polygonal meshes and second-order polynomials.}
\label{errore_h}
\end{figure}
\begin{figure}
\centering
\subfloat[$\|\bf u - \bf u_h\|_{\dGe}$ and $\|\phi - \phi_h\|_\dGa$ vs.~$p$ at $T=1$]{\includegraphics[scale=.45]{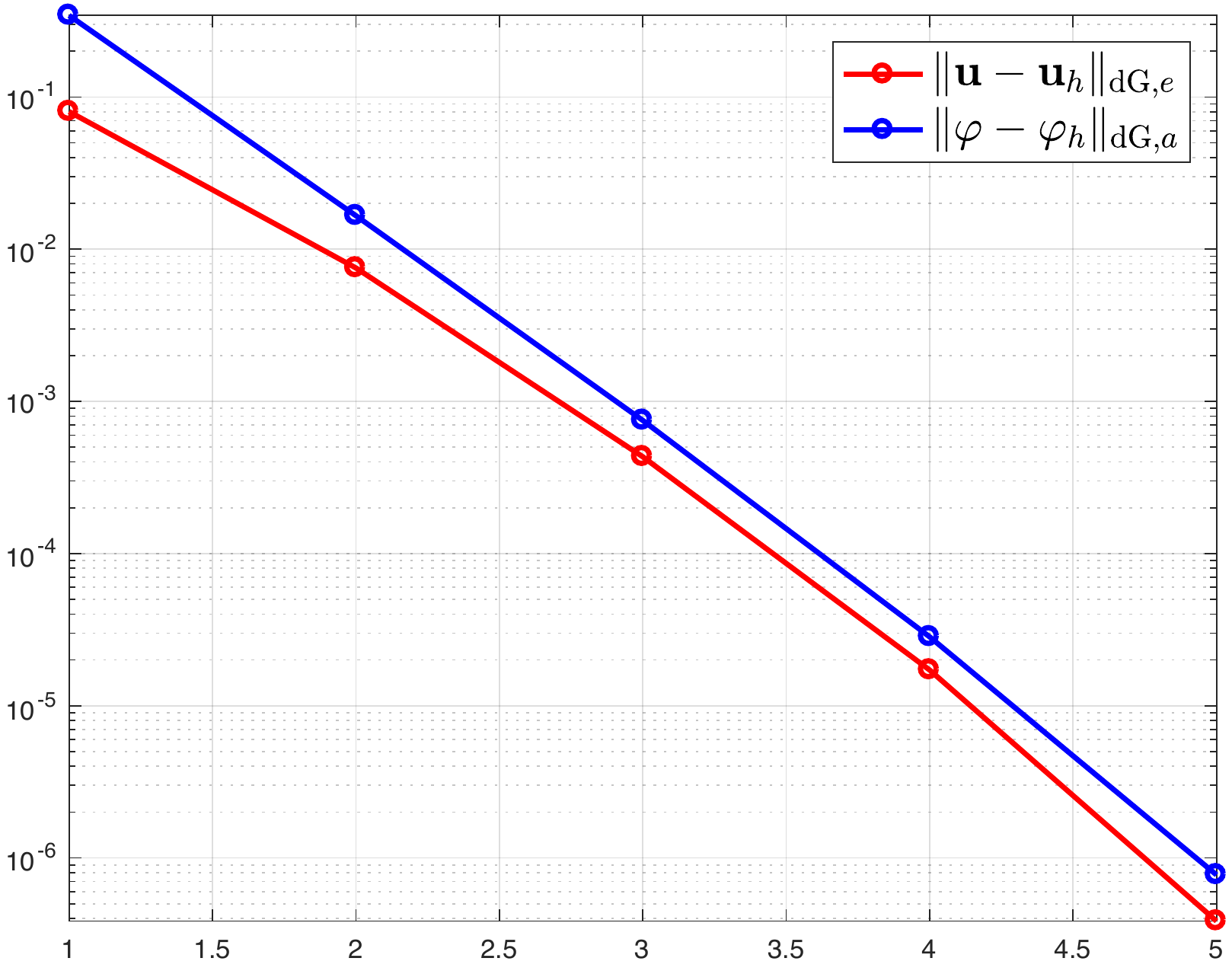}}
\hspace{.01cm}
\subfloat[$\|\bf u - \bf u_h\|_\Oe$ and $\|\phi - \phi_h\|_\Oa$ vs.~$p$ at $T=1$]{\includegraphics[scale=.45]{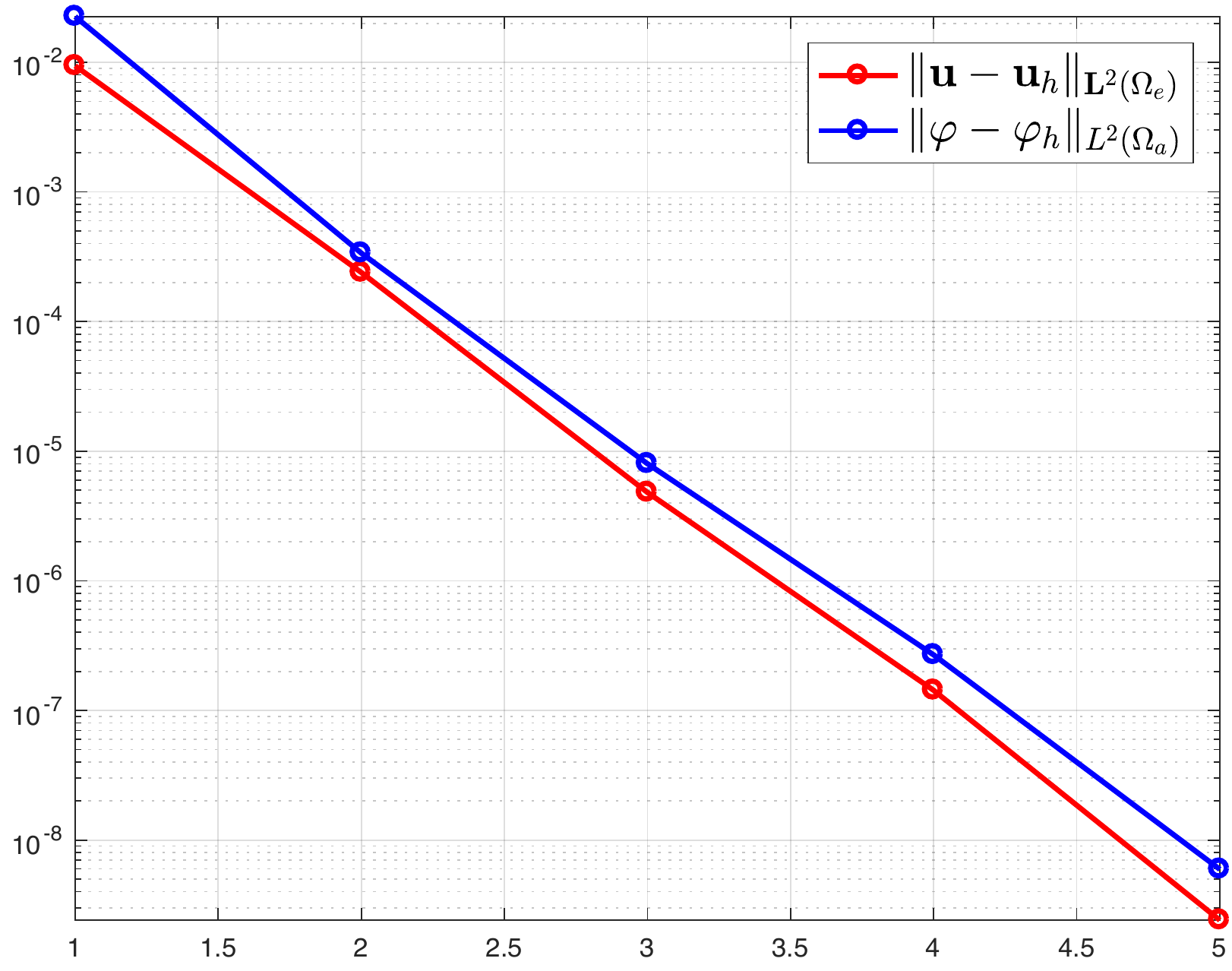}}
\caption{Test case1. dG-error and $L^2$-error vs.~$p$ for $p$ ranging from 1 to 5 and a mesh given by 300 polygons.}
\label{errore_p}
\end{figure}
\subsection{Test case 2}\label{sec:nonhom_bc}
We now choose the right-hand sides $\bf f_e$ and $f_a$ so that the exact solution is given by
\begin{equation}\label{nonhom_bc}
\begin{aligned}
\bf u(x,y;t) = \left(\cos\Big(\frac{4\pi x}{c_p}\Big), \, \cos\Big(\frac{4\pi x}{c_s}\Big)\right)\cos( 4\pi t), \ \
\phi(x,y;t)  = \sin(4\pi x)\sin(4\pi t),
\end{aligned}
\end{equation}
where 
$c_p = \sqrt{\frac{\lambda+2\mu}{\rho_e}}$ and $c_s = \sqrt{\frac{\mu}{\rho_e}}$ are the velocities of pressure and shear waves in the elastic domain, respectively. The same test has been carried out in \cite{monkola-sanna} using a Spectral Element discretization; the choice of material parameters is also the same as in the previous test case. In this case, on $\Gamma_\mr I$, both the traction $\bm\sigma(\bf u)\bf n_e$ and the acoustic pressure $-\rho_a\dot\phi\bf n_e$ vanish; on the other hand, we have $\de\phi/\de\bf n_a = - \dot{\bf u}\cdot \bf n_a = 4\pi\sin(4\pi t)$. The timestep is, again, set to $\Delta t = 10^{-4}$; on the other hand, the final time is in this case set to $T=0.8$, to ensure that none of the two unknowns $\bf u$ and $\phi$ be identically zero when dG- and $L^2$-errors are computed.
%
%\subsubsection*{\boldmath $h$-refinement}

Figure \ref{errore_h_NR} shows convergence results in the dG- and $L^2$-norms respectively, for four nested, sequentially refined polygonal meshes, when polynomials of uniform degree $p=2$ are employed. The numerical results concerning the dG-error again show asymptotic convergence rates matching those predicted by estimate \eqref{errore'}. Also, the $L^2$-error convergence rates turn out to be slightly higher than $h^{p+1}$ both for $\bf u$ and for $\phi$; in the latter case, this difference is more remarkable.
%\subsubsection*{\boldmath $p$-refinement}

Figures \ref{errore_p_NR} shows convergence results in a semilogarithmic scale, in the dG- and $L^2$-norms respectively, for a fixed mesh given by 300 elements and a uniform polynomial degree ranging from 1 to 5. Again, the error undergoes an exponential decay. Notice that, concerning the $L^2$-error on $\bf u$ (Figure \ref{errore_p_NR}b), the convergence rate decreases when passing from polynomial degree 4 to 5: in both cases the $L^2$-error is on the order of $10^{-7}$. This behavior is related to the choice of the timestep $\Delta t$, set to $10^{-4}$; indeed, when a leap-frog time discretization is employed, the error is expected to converge in $\Delta t^2$. In our case, $\Delta t^2 = 10^{-8}$, which is only one order of magnitude lower than the $L^2$-error for $p=4$ and $p=5$. Decreasing the timestep to $\Delta t = 10^{-5}$ allows to recover the expected convergence.
\begin{figure}
\centering
\subfloat[$\|\bf u - \bf u_h\|_{\dGe}$ and $\|\phi - \phi_h\|_\dGa$ vs.~$h$ at $T=0.8$]{\includegraphics[scale=.45]{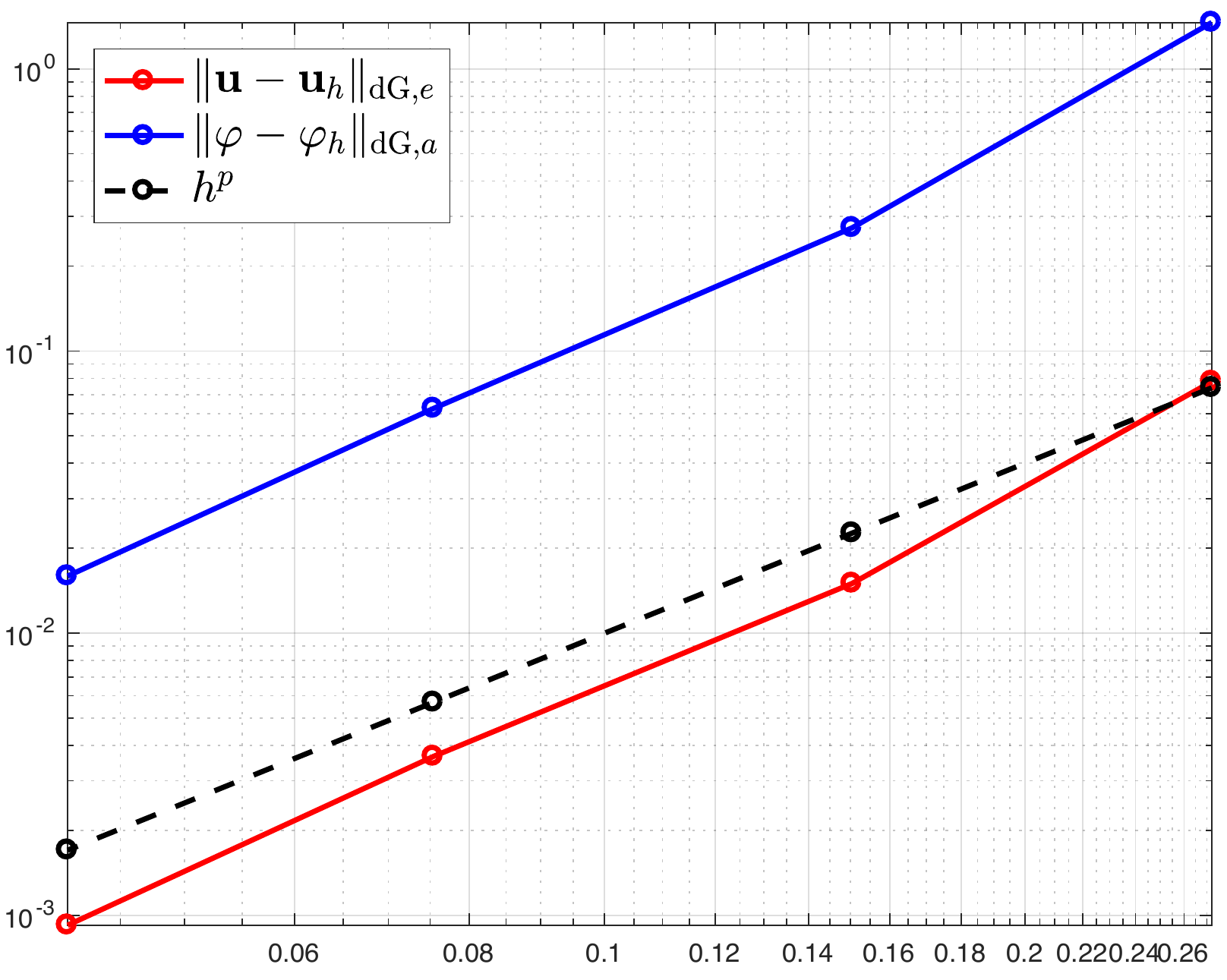}}
\hspace{.01cm}
\subfloat[$\|\bf u - \bf u_h\|_{\Oe}$ and $\|\phi - \phi_h\|_{\Oa}$ vs.~$h$ at $T=0.8$]{\includegraphics[scale=.45]{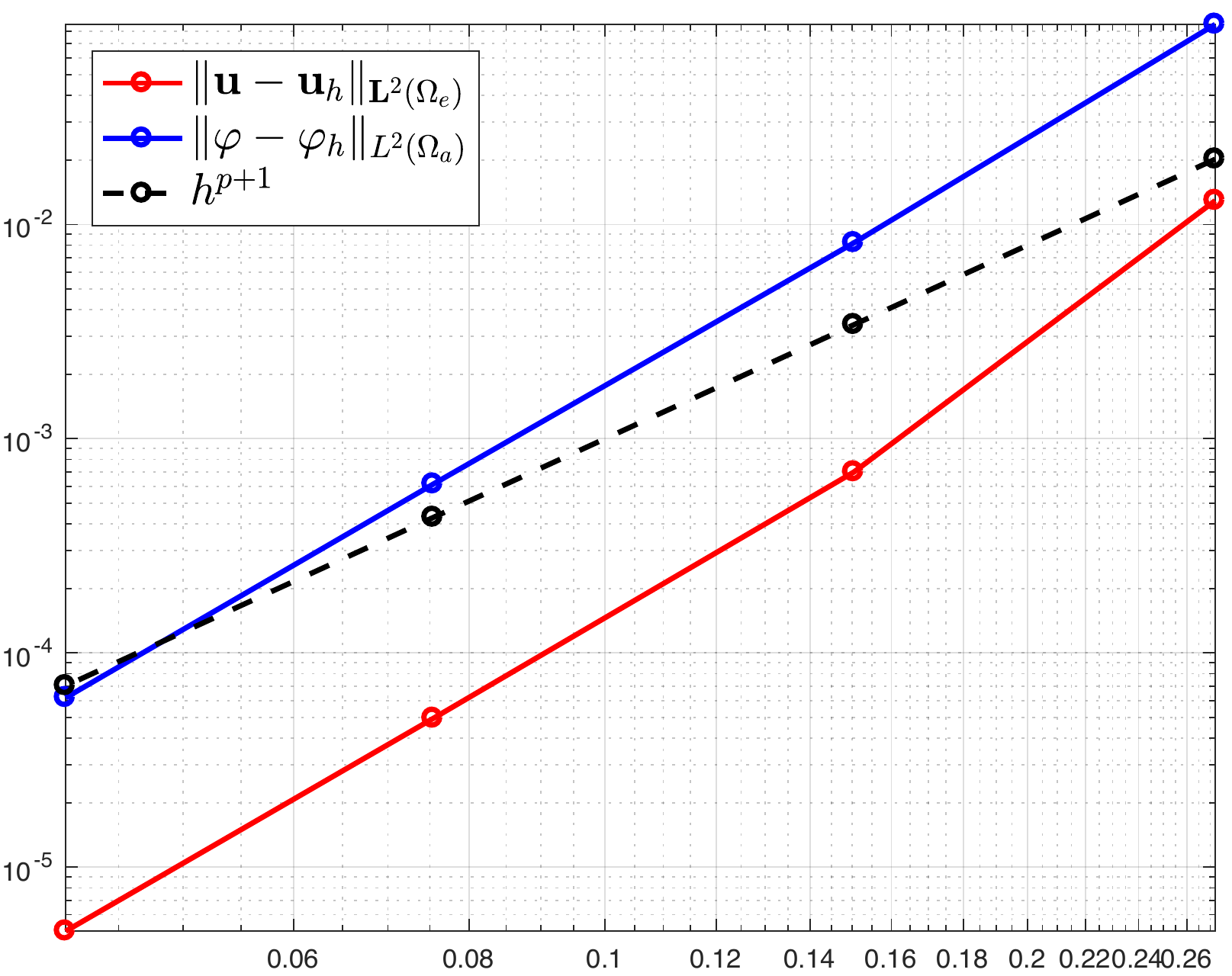}}
\caption{Test case 2. dG-error and $L^2$-error vs.~$h$ for four sequentially refined polygonal meshes and second-order polynomials.}
\label{errore_h_NR}
\end{figure}
\begin{figure}
\centering
\subfloat[$\|\bf u - \bf u_h\|_{\dGe}$ and $\|\phi - \phi_h\|_\dGa$ vs.~$p$ at $T=0.8$]{\includegraphics[scale=.45]{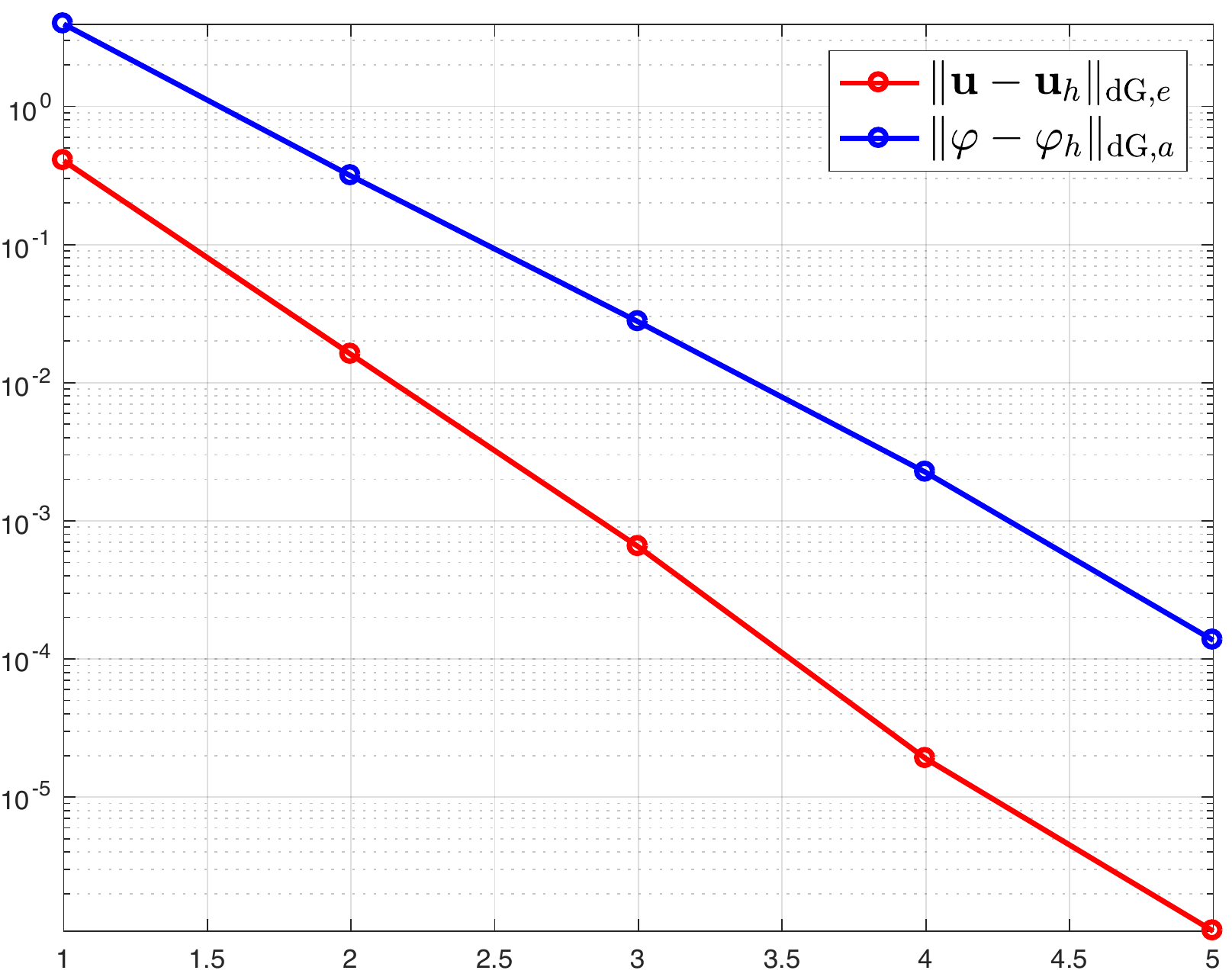}}
\hspace{.01cm}
\subfloat[$\|\bf u - \bf u_h\|_{\Oe}$ and $\|\phi - \phi_h\|_{\Oa}$ vs.~$p$ at $T=0.8$]{\includegraphics[scale=.45]{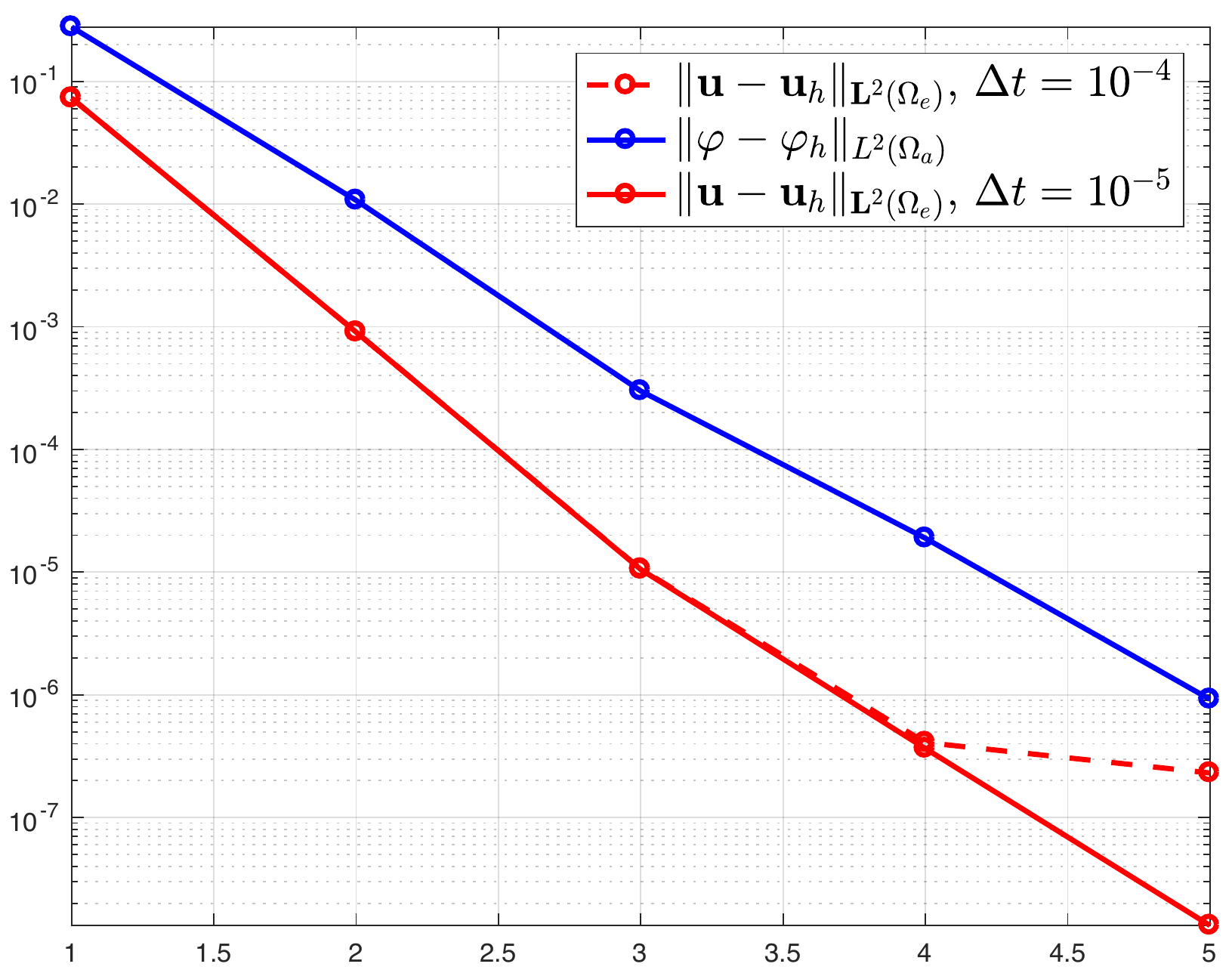}}
\caption{Test case 2. dG-error and $L^2$-error vs.~$p$ for $p$ ranging from 1 to 5 and a mesh given by 300 polygons.}
\label{errore_p_NR}
\end{figure}
\subsection{Test case 3: a physical example}
As a further numerical experiment, we simulate a seismic source. In particular, we suppose that the system is excited \emph{only} by a Ricker wavelet, i.e., by the following point source load placed in the acoustic domain:
\begin{equation}\label{dirac_source}
f_a(\bf x,t) = -2\pi a\left(1-2\pi a(t-t_0)^2\right)e^{-\pi a(t-t_0)^2} \delta(\bf x - \bf x_0),\ \ \bf x_0 \in \Oa,\ t_0 \in (0,T],
\end{equation}
where $\bf x \equiv (x,y)$, $\bf x_0 \equiv (x_0,y_0)$ is a given point in $\Oa$, and $\delta$ is the Dirac distribution (cf.~Figure \ref{f_a} for a representation of the time factor in \eqref{dirac_source}). All initial conditions, as well as the body force $\bf f_e$, are set to zero. 
The Dirac distribution in $\bf x_0$ is approximated numerically by a Gaussian distribution centered at $\bf x_0$. We consider the following values of the material parameters: $\rho_e = 2.5$, $\rho_a = 1$, $\mu = 10$, $\lambda = 20$, $c=1.5$; also, in \eqref{dirac_source}, we choose $\bf x_0 = (0.2, 0.5)$, $t_0 = 0.1$, and $a=576$. We employ here a polygonal mesh of 5000 elements, corresponding to a meshsize $h \simeq 0.04$, a uniform polynomial degree $p=3$, and a timestep $\Delta t = 10^{-5}$. The final time is set to $T=1$.

%Figure \ref{u_point} shows the elastic displacement at three points in the elastic domain $\Oe$, placed at increasing distances from the interface $\Gamma_\mr I$. It can be clearly observed that the vertical displacement is very close to zero in all of the three cases. This is a natural consequence of the coupling: it propagates only longitudinal
%stresses through the elasto-acoustic interface, since fluids cannot sustain shear stresses. Thus, since the monitored elastic points are horizontally aligned with the hypocentre, the only nonzero component of the displacement is the horizontal one. Also, as one expects, the farther the elastic point is from the interface, the more its horizontal motion is delayed in time. In the case of point $(-0.9,0.5)$, which is the farthest from $\Gamma_{\mr I}$ among the three, one can also observe very small oscillations around zero in the vertical displacement, when approaching time $t=0.9$. This is due to the fact that waves are completely reflected on the boundary, because of homogeneous Dirichlet boundary conditions. Such reflections are also well represented in Figure \ref{snapshots}, where the numerical solution is displayed at time $t=0.5$.
%
Figure \ref{snapshots} shows the numerical solution (horizontal and vertical elastic displacements, and acoustic potential) at time $t=0.5$. The vertical displacement, displayed in Figure \ref{snapshots}b, turns out to be very close to zero in a large elastic subregion, except near the boundary, where small reflected wavefronts can be detected, because of homogeneous Dirichlet boundary conditions. This behavior is due to the fact that the seismic source is placed close enough to the interface $\Gamma_\mr I$, so that the effects of reflected waves in the elastic region are not observed for a certain time, and hence only the coupling effects are visible (only longitudinal stresses are propagated through the elasto-acoustic interface, since fluids cannot sustain shear stresses). Nevertheless, after a certain time, elastic waves are reflected, which gives rise to a nonzero vertical displacement. Concerning the acoustic region, spherical wavefronts generated by the point source load can be clearly observed in Figure \ref{snapshots}c; again, waves are reflected on the boundary for the same reason as before. 
%%%%%%%%%
\begin{figure}
\includegraphics[scale=.2]{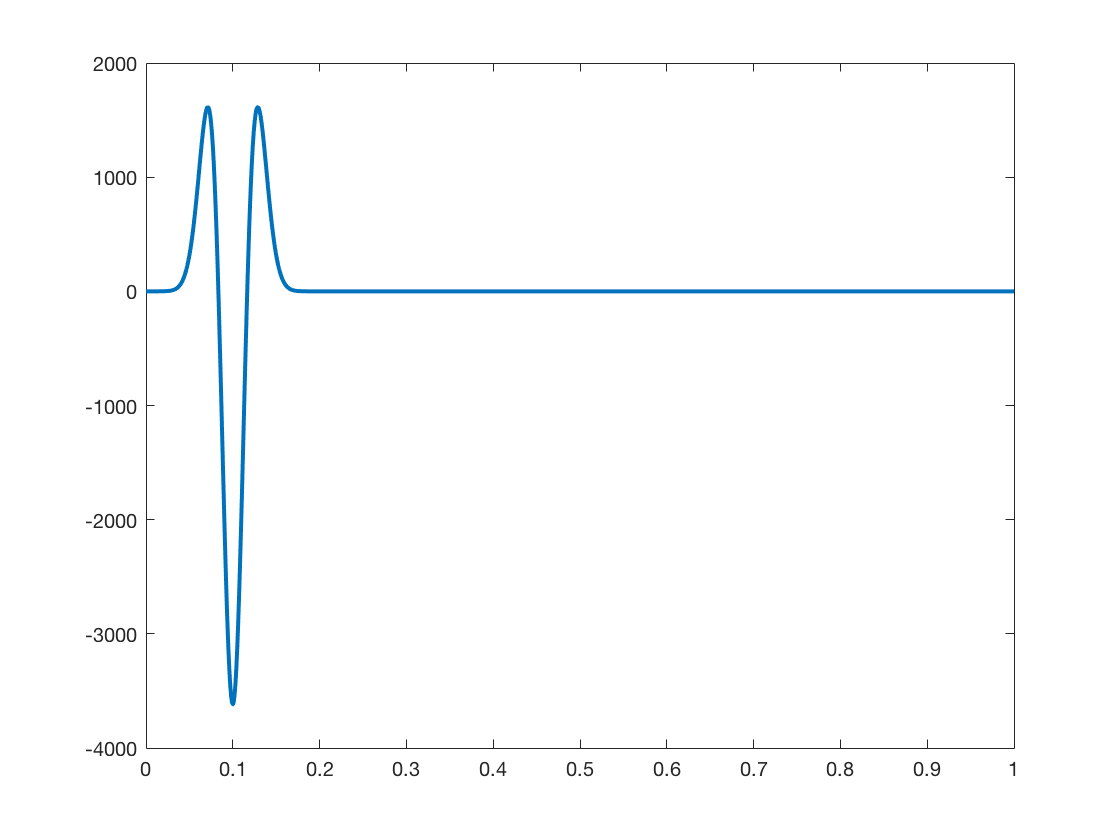}
\caption{$t \mapsto -2\pi a\left(1-2\pi a(t-t_0)^2\right)e^{-\pi a(t-t_0)^2}$ for $a=576$ and $t_0=0.1$.}
\label{f_a}
\end{figure}
%%%%%%%%%%%%%%%%%%%%%
%\begin{figure}
%\centering
%\subfloat[$t \mapsto \bf u_h(-0.01,0.5;t)$]{\includegraphics[scale=.225]{u_at_00105}}
%\hspace{.01cm}
%\subfloat[$t \mapsto \bf u_h(-0.3,0.5;t)$]{\includegraphics[scale=.225]{u_at_0305}}
%\hspace{.01cm}
%\subfloat[$t \mapsto \bf u_h(-0.9,0.5;t)$]{\includegraphics[scale=.225]{u_at_0905}}
%\caption{Elastic displacement at given points in $\Oe$.}
%\label{u_point}
%\end{figure}
%%%%%%%%%%%
\begin{figure}
\centering
\subfloat[Horizontal elastic displacement at $t=0.5$]{\includegraphics[scale=.4]{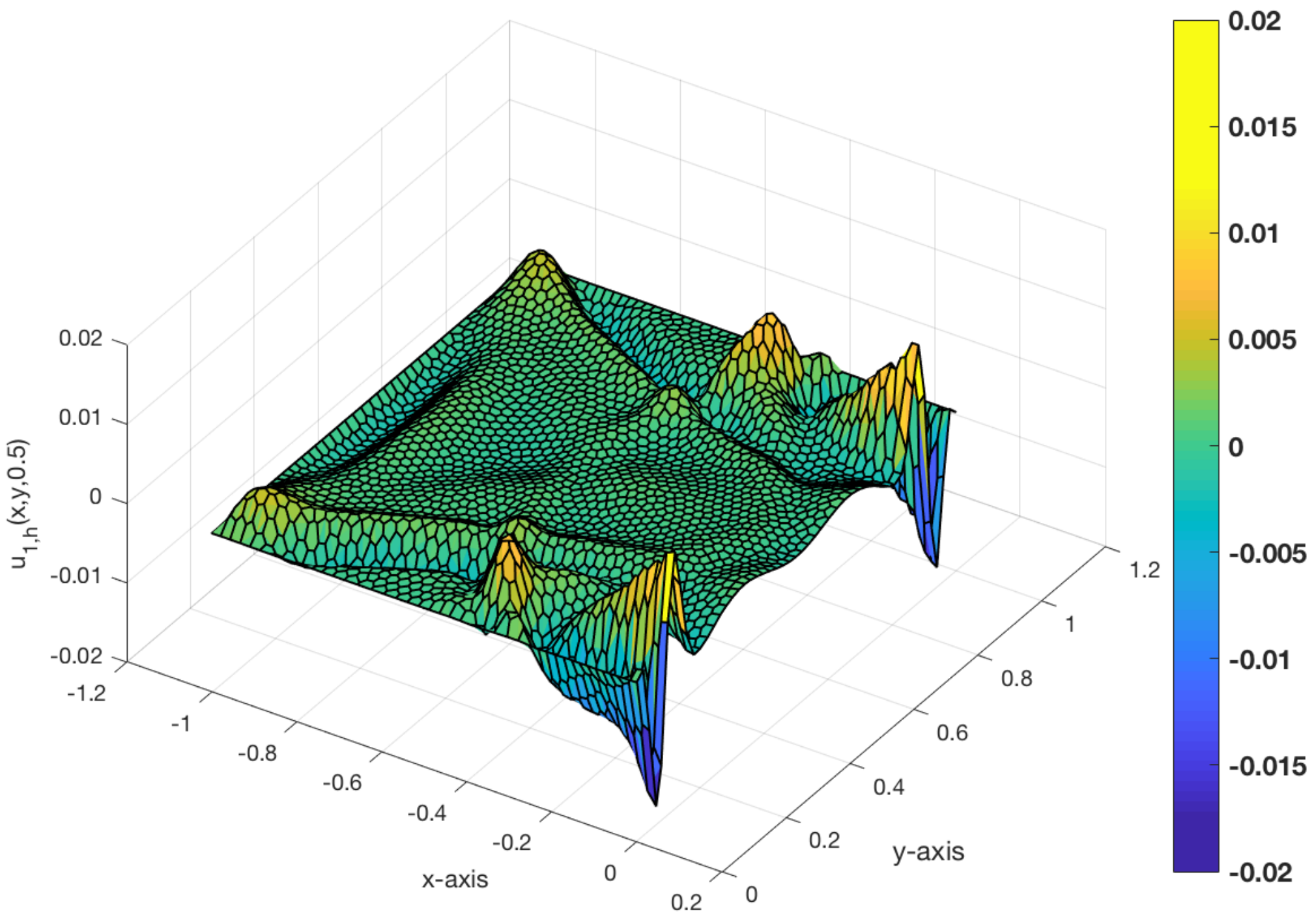}}
\hspace{.02cm}
\subfloat[Vertical elastic displacement at $t=0.5$]{\includegraphics[scale=.4]{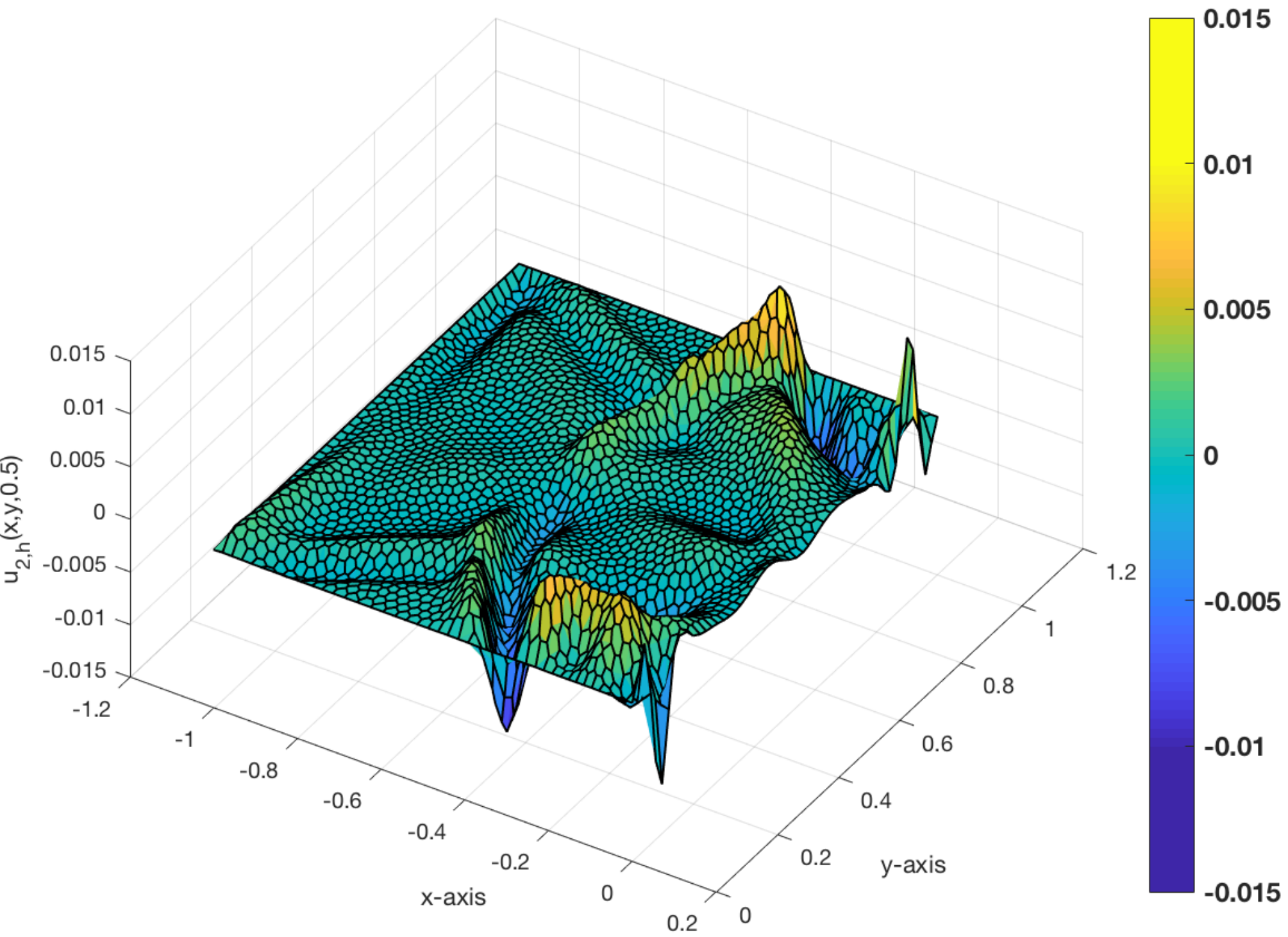}}
\hspace{.02cm} % \\ \vspace{1cm}
\subfloat[Acoustic potential at $t=0.5$]{\includegraphics[scale=.4]{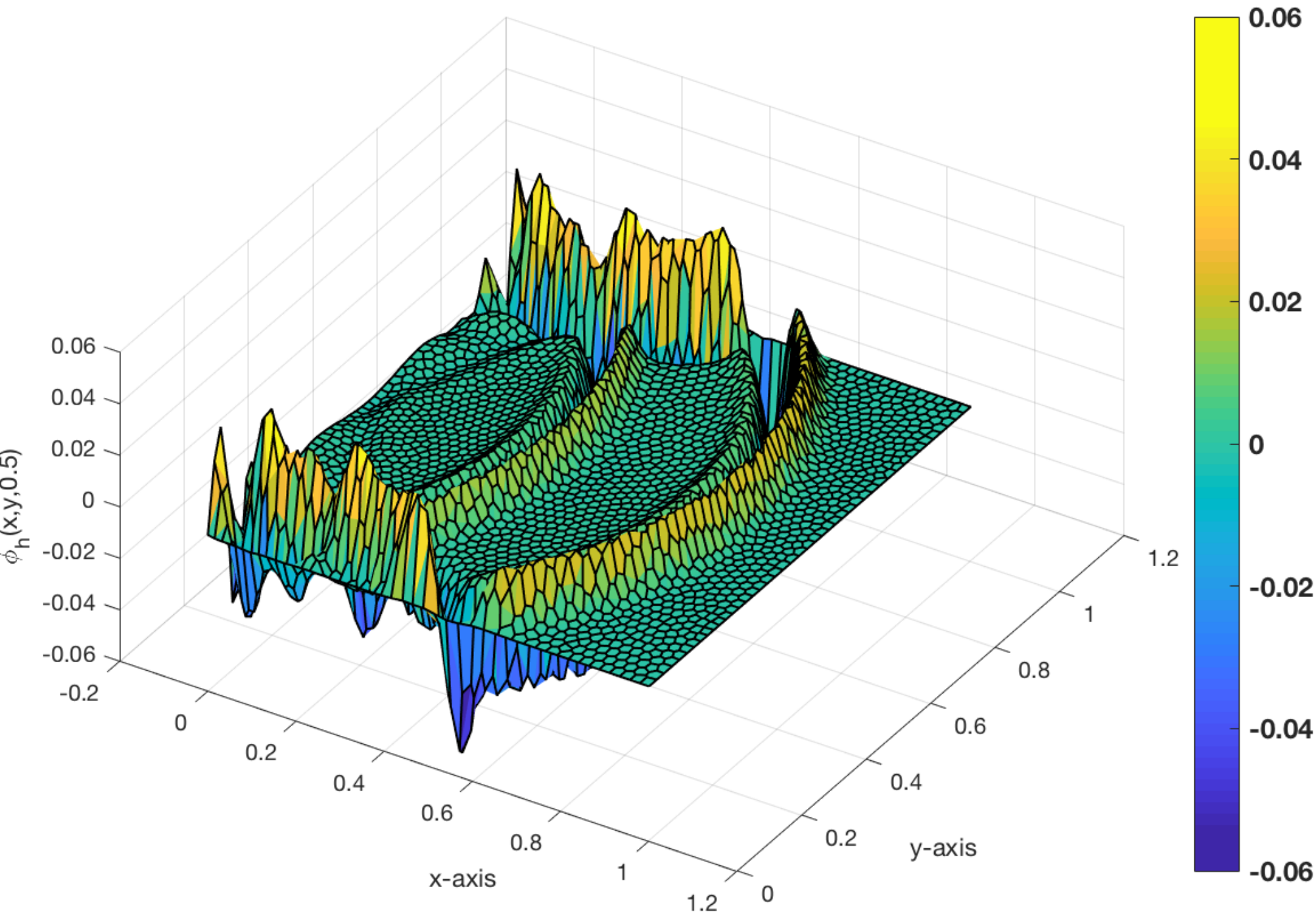}}
\caption{Numerical solution at $t=0.5$.}
\label{snapshots}
\end{figure}

%% ---------------------------------------------------
\appendix
\section{}\label{app}
     \renewcommand{\thelemma}{\Alph{section}\arabic{lemma}}
\begin{lemma}\label{lemma_stab}
The following inequalities hold:
\begin{subequations}
\begin{alignat}{2}
\|\eta^{-\nf{1}{2}}\av{\bm\sigma_h(\bf v)}\|_{{\cc F}_h^e} & \lesssim \frac{1}{\sqrt{\alpha}} \|\bb C^{\nf{1}{2}}\bm\eps_h(\bf v)\|_\Oe &\quad& \forall \bf v \in \bf V_{\! h}^e,\label{prima}\\
\|\chi^{-\nf{1}{2}}\av{\rho_a\grad_h\psi}\|_{{\cc F}_h^a} & \lesssim \frac{1}{\sqrt{\beta}} \|\rho_a^{\nf{1}{2}}\grad_h\psi\|_\Oa &\quad& \forall \psi \in V_h^a,\label{seconda}
\end{alignat}
\end{subequations}
where $\alpha$ and $\beta$ are the stability parameters appearing in the definition of stabilization functions \eqref{stabiliz.a}--\eqref{stabiliz.b}.
\end{lemma}
\begin{proof} We only prove \eqref{prima}, the arguments for showing \eqref{seconda} being completely analogous.

Recall that the following trace-inverse inequality holds for simplices \cite[p.~25]{cangiani-book}: given a simplex $T\subset \bb R^d$ and a polynomial degree $p \ge 1$, for all $v\in \mathscr P_p(T)$ there is a real number $C > 0$ independent of the discretization parameters such that
\begin{equation}\label{inv_for_simplex}
\|v\|_F^2 \le C p^2\frac{ |F|}{|T|}\|v\|_T^2.
\end{equation}

Owing to \eqref{inv_for_simplex}, the definition \eqref{C_k} of $\wl{\bb C}_\k$, the definition \eqref{stabiliz.a} of $\eta$, and Assumption \ref{uno}, for any $\bf v \in \bf V_{\! h}^e$ we obtain
$$
\begin{aligned}
\|\eta^{-\nf{1}{2}}\av{\bm\sigma_h(\bf v)}\|_{{\cc F}_h^e}^2
 & \le \sum_{\k \in \cc T_h^e} \sum_{F\subset \de\k} \wl{\bb C}_\k \| \eta^{-\nf{1}{2}}
 \bb C^{\nf12} \bm\eps (\bf v)\|_{F}^2 \\
 &  \lesssim \sum_{\k \in \cc T_h^e} \sum_{F\subset \de\k} \eta^{-1}\wl{\bb C}_\k p_{e,\k}^2
 \frac{ |F|}{ |\k_\flat^F|} \|\bb C^{\nf12} \bm\eps(\bf v)\|_{\k_\flat^F}^2 \lesssim \frac{1}{\alpha}
 \|\bb C^{\nf12}\bm\eps_h(\bf v)\|_{\Oe}^2.
\end{aligned}
$$
\end{proof}
\begin{lemma}\label{bounds}
For any $\W = (\bf v,\psi) \in C^1([0,T];\bf V_{\! h}^e) \times C^1([0,T];V_h^a)$, it holds
\begin{equation}\label{eq:bounds}
\begin{aligned}
\|\W\|_\E^2 - 2\left( 
\la \av{\bm\sigma_h(\bf v)},\jm{\bf v}\ra_{{\cc F}_h^e} + \la\av{\rho_a\grad_h\psi},\jm{\psi}\ra_{{\cc F}_h^a}
\right) &\lesssim \|\W\|_\E^2,
\\
\|\W\|_\E^2 - 2\left( 
\la \av{\bm\sigma_h(\bf v)},\jm{\bf v}\ra_{{\cc F}_h^e} + \la\av{\rho_a\grad_h\psi},\jm{\psi}\ra_{{\cc F}_h^a}
\right) &\gtrsim \|\W\|_\E^2,
\end{aligned}
\end{equation}
\end{lemma}
\begin{proof}
The first bound follows from the Cauchy--Schwarz inequality, the definition \eqref{energy_norm} of the energy norm, and Lemma \ref{lemma_stab}:
\begin{align*}
\|\W\|_\E^2 \ - & \ 2\left( 
\la \av{\bm\sigma_h(\bf v)},\jm{\bf v}\ra_{{\cc F}_h^e} + \la\av{\rho_a \grad_h\psi},\jm{\psi}\ra_{{\cc F}_h^a}
\right) 
\\
& \lesssim \  \|\W\|_\E^2 + \|\eta^{-\nf12}\av{\bm\sigma_h(\bf v)}\|_{{\cc F}_h^e} \|\eta^{\nf12}\jm{\bf v}\|_{{\cc F}_h^e}
 + \|\chi^{-\nf12}\av{\rho_a\grad_h\psi}\|_{{\cc F}_h^a} \|\chi^{\nf12}\jm{\psi}\|_{{\cc F}_h^a} \\
& \lesssim  \|\W\|_\E^2 + \frac{1}{\sqrt{\alpha}} \|\bb C^{\nf12}\bm\eps_h(\bf v)\|_\Oe \|\bf v\|_\dGe + \frac{1}{\sqrt{\beta}}
\|\rho_a^{\nf12}\grad_h\psi\|_\Oa \|\psi\|_\dGa \\
& \lesssim \|\W\|_\E^2 + \|\W\|_\dG^2 \lesssim \|\W\|_\E^2,
\end{align*}
where we have set $\| \W \|_\dG^2 = \|\bf v\|_{\dG,e}^2 + \|\psi\|_{\dG,a}^2$.
To prove the second bound, it suffices to show that
\begin{equation}\label{bound_dG}
\|\W\|_\dG^2 - 2\left( 
\la \av{\bm\sigma_h(\bf v)},\jm{\bf v}\ra_{{\cc F}_h^e} + \la\av{\rho_a\grad_h\psi},\jm{\psi}\ra_{{\cc F}_h^a}
\right) \gtrsim \|\W\|_\dG^2.
\end{equation}
Indeed, by the definition \eqref{energy_norm} of the energy norm and \eqref{bound_dG},
%the left-hand side of \eqref{eq:bounds} is equal to
\begin{multline*}
\|\W\|_\E^2 - 2\left( 
\la \av{\bm\sigma_h(\bf v)},\jm{\bf v}\ra_{{\cc F}_h^e} + \la\av{\rho_a\grad_h\psi},\jm{\psi}\ra_{{\cc F}_h^a}
\right)   \\
= \|\rho_e^{\nf12}\dot{\bf v}\|_\Oe^2 + \|\rho_e^{\nf12}\z \bf v\|_\Oe^2 + \|c^{-1}\rho_a^{\nf12}\dot\psi\|_\Oa^2 + \|\W\|_\dG^2 
- 2\left( 
\la \av{\bm\sigma_h(\bf v)},\jm{\bf v}\ra_{{\cc F}_h^e} + \la\av{\rho_a\grad_h\psi},\jm{\psi}\ra_{{\cc F}_h^a}
\right)
\\
\gtrsim \|\rho_e^{\nf12}\dot{\bf v}\|_\Oe^2 + \|\rho_e^{\nf12}\z \bf v\|_\Oe^2 + \|c^{-1}\rho_a^{\nf12}\dot\psi\|_\Oa^2 + \|\W\|_\dG^2 = \|\W\|_\E^2.
\end{multline*}
Thus, we next show that \eqref{bound_dG} holds provided the stability parameters $\alpha$ and $\beta$ are chosen large enough.
To this purpose, using Young's inequality we infer that, for any $\delta,\epsi > 0$,
\begin{align*}
%\begin{aligned}
\la \av{\bm\sigma_h(\bf v)},\jm{\bf v}\ra_{{\cc F}_h^e} &\le \sum_{F\in{\cc F}_h^e} \|\eta^{-\nf12}\av{\bm\sigma_h(\bf v)}\|_F \|\eta^{\nf12}\jm{\bf v}\|_F  \le\frac{1}{2\delta}\|\eta^{-\nf12}\av{\bm\sigma_h(\bf v)}\|_{{\cc F}_h^e}^2 + \frac{\delta}{2}\|\eta^{\nf12}\jm{\bf v}\|_{{\cc F}_h^e}^2, \\
%\end{aligned}
%
%\begin{aligned}
\la \av{\rho_a \grad_h \psi},\jm{\psi}\ra_{{\cc F}_h^a} & \le \sum_{F\in{\cc F}_h^a}
\|\chi^{-\nf12}\av{\rho_a\grad_h\psi}\|_F \|\chi^{\nf12}\jm{\psi}\|_F
 \le\frac{1}{2\epsi}\|\chi^{-\nf12}\av{\rho_a \grad_h \psi}\|_{{\cc F}_h^a}^2 + \frac{\epsi}{2}\|\chi^{\nf12}\jm{\psi}\|_{{\cc F}_h^a}^2.
%\end{aligned}
\end{align*}
Hence, from the definition of the $\|{\cdot}\|_{\dG,e}$- and $\|{\cdot}\|_{\dG,a}$-norms on $\bf V_{\! h}^e$ and $V_h^a$, it follows that
\begin{multline*}
\|\W\|_\dG^2 - 2\left( 
\la \av{\bm\sigma_h(\bf v)},\jm{\bf v}\ra_{{\cc F}_h^e} + \la\av{\rho_a\grad_h\psi},\jm{\psi}\ra_{{\cc F}_h^a}
\right)
\\
 \ge \|\bb C^{\nf12}\bm\eps(\bf v)\|_\Oe^2 + \|\rho_a^{\nf12}\grad_h\psi\|_\Oa^2 + \left(1-{\delta}\right)
\|\eta^{\nf12}\jm{\bf v}\|_{{\cc F}_h^e}^2 - \frac{1}{\delta} \|\eta^{-\nf12}\av{\bm\sigma_h(\bf v)}\|_{{\cc F}_h^e}^2
\\
+ \left(1-{\epsi}\right)\|\chi^{\nf12}\jm{\psi}\|_{{\cc F}_h^a}^2 - \frac{1}{\epsi}\|\chi^{-\nf12}\av{\rho_a \grad_h \psi}\|_{{\cc F}_h^a}^2
\\
\ge \left(1-\frac{C_1}{\alpha\delta}\right)\|\bb C^{\nf12}\bm\eps_h(\bf v)\|_\Oe^2 +
\left(1-\frac{C_2}{\beta\epsi}\right) \|\rho_a^{\nf12}\grad_h\psi\|_\Oa^2
\\
+ \left(1-{\delta}\right) \|\eta^{\nf12}\jm{\bf v}\|_{{\cc F}_h^e}^2 + \left(1-{\epsi}\right)\|\chi^{\nf12}\jm{\psi}\|_{{\cc F}_h^a}^2,
\end{multline*}
where in the last bound we have applied Lemma \ref{lemma_stab} with hidden constants $C_1$ and $C_2$.
Then \eqref{bound_dG} follows by choosing, for instance, $\delta=\epsi=\nf12$ and $\alpha \ge 4C_1$, $\beta \ge 4C_2$.
\end{proof}

 \bibliographystyle{elsarticle-num} 
 \bibliography{references}
\end{document}